\documentclass[]{article}

\usepackage{arxiv}

\usepackage[utf8]{inputenc} 
\usepackage[T1]{fontenc}    
\usepackage{hyperref}       
\usepackage{url}            
\usepackage{booktabs}       
\usepackage{amsfonts}       
\usepackage{nicefrac}       
\usepackage{microtype}      
\usepackage{lipsum}		
\usepackage{graphicx}
\usepackage{natbib}
\usepackage{doi}
\usepackage{framed,multirow}

\usepackage{amssymb}
\usepackage{latexsym}

\usepackage{url}
\usepackage{xcolor}
\definecolor{newcolor}{rgb}{.8,.349,.1}

\usepackage{amsmath, amsthm}
\usepackage{cleveref}
\usepackage{pdflscape}
\usepackage{physics}
\usepackage{bm}
\usepackage[colorinlistoftodos]{todonotes}
\usepackage{arydshln}
\usepackage{bbm}
\usepackage{enumerate}

\theoremstyle{definition}
\newtheorem{thm}{Theorem}[section]
\newtheorem{thm*}{Theorem}[section]

\newtheorem{mydef}{Definition}[section]
\newtheorem{rmk}{Remark}
\newtheorem{ex}{Example}

\newcommand{\zl}{z^{[\ell]}}
\newcommand{\lam[1]}{\lambda^{[#1]}}
\newcommand{\yy}{\mathbf{y}}
\newcommand{\yyy[1]}{\mathbf{y}^{[#1]}}
\newcommand{\YY}{\mathbf{Y}}

\newcommand{\ts}{\tilde{s}}
\newcommand{\ta[1]}{\tilde{a}^{[#1]}}
\newcommand{\Dt}{\Delta{t}}
\newcommand{\Dx}{\Delta{x}}
\newcommand{\F}{\mathcal{F}}
\newcommand{\Fl[1]}{\mathbf{\mathcal{F}}^{[#1]}}
\newcommand{\cc[1]}{\mathbf{c}^{[#1]}}
\newcommand{\tc[1]}{\tilde{c}^{[#1]}}
\newcommand{\tcc[1]}{\tilde{\mathbf{c}}^{[#1]}}
\newcommand{\AAA[1]}{\mathbf{A}^{[#1]}}
\newcommand{\bb[1]}{\mathbf{b}^{[#1]}}
\newcommand{\tb[1]}{\tilde{b}^{[#1]}}
\newcommand{\tbb[1]}{\tilde{\mathbf{b}}^{[#1]}}

\newcommand{\tA[1]}{\tilde{\mathbf{A}}^{[#1]}}
\newcommand{\ee}{\mathbbm{1}}
\newcommand{\osgarkY}[3]{\mathbf{Y}^{[#2]}_{#1,#3}}
\newcommand{\bbS}[2]{\mathbb{S}_{#1}^{#2}}

\newcommand{\Nop}{N}
\newcommand{\pphi}[2]{\varphi^{[#1]}_{#2}}
\newcommand{\aalpha}{\bm{\alpha}}
\newcommand{\aaalpha}[2]{\alpha_{#1}^{[#2]}}
\newcommand{\Zero}{\bm{0}}
\def\OS22b{\mu}

   \title{Fractional-Step Runge--Kutta Methods: Representation and
  Linear Stability Analysis  \thanks{
  	This work was supported by the National Sciences and Engineering Research Council of Canada through its Discovery Grant program.}}
\author{ \href{}{\hspace{1mm}Raymond J. Spiteri}\\
	Department of Computer Science\\
	University of Saskatchewan, Saskatoon, SK, Canada\\
	\\
	\texttt{spiteri@cs.usask.ca} \\
	\And
	\href{}{\hspace{1mm}Siqi Wei}\\
	Department of Mathematics and Statistics \\ 
	University of Saskatchewan, Saskatoon, SK, Canada \\\\
	\texttt{siqi.wei@usask.ca} \\
}

\date{}

\renewcommand{\shorttitle}{Fractional-Step Runge--Kutta Methods}

\begin{document}
\maketitle

\begin{abstract}
	Fractional-step methods are a popular and powerful
	divide-and-conquer approach for the numerical solution of
	differential equations. When the integrators of the fractional steps
	are Runge--Kutta methods, such methods can be written as generalized
	additive Runge--Kutta (GARK) methods, and thus the representation
	and analysis of such methods can be done through the GARK
	framework. We show how the general Butcher tableau representation
	and linear stability of such methods are related to the coefficients
	of the splitting method, the individual sub-integrators, and the
	order in which they are applied.  We use this framework to explain
	some observations in the literature about fractional-step methods
	such as the choice of sub-integrators, the order in which they are
	applied, and the role played by negative splitting coefficients in
	the stability of the method.
\end{abstract}



\section{Introduction}

The right-hand side of an explicit ordinary differential equation
(ODE) is often additively comprised of terms having different
character, e.g., linear vs.~nonlinear, stiff vs.~non-stiff, or derived
from different physical phenomena such as advection vs.~reaction
vs.~diffusion. In such cases, it is natural (and often advantageous)
to consider a splitting approach that treats the different terms with
different numerical methods. In this way, the terms can be treated in
specialized ways, potentially leading to efficient methods or ones
with special properties such as symplecticity or strong
stability. Indeed, when using different libraries as black boxes for
simulations of different parts of the system as in co-simulation (see,
e.g.,~\cite{Gomes2018-coSimulation} and references therein), there may
be no choice but to treat the parts separately. Similarly, it may not be
feasible to solve certain problems in a monolithic sense.


Divide-and-conquer approaches to solving ODEs date at least as far
back as Sophus Lie in the 1870s~\cite{Lie1888}. They have a long and
diverse history, and because of this, they have been known by many
names and have subtle differences between them. Such names include
operator splitting, time splitting, split-step methods, dimensional
splitting, locally one-dimensional (LOD) methods, alternating
direction implicit (ADI) methods, approximate matrix factorization
(AMF) methods, and additive methods (and their most popular special
case, implicit-explicit (IMEX) methods);~see,
e.g.,~\cite{Hundsdorfer2003,McLachlan2002,Glowinski2017} and
references therein. When used to solve differential-algebraic
equations, such as those arising from the incompressible
Navier--Stokes equations
, they are also called projection methods; see,
e.g,.~\cite{Guermond2006} and references therein.

Consider the initial-value problem (IVP) for an $\Nop$-addi\-tive\-ly
split ordinary differential equation
\begin{equation} \label{cauchy_problem}
\dv{\yy}{t} = \F(t,\yy) = \sum\limits_{\ell=1}^\Nop \Fl[\ell](t,\yy),
\qquad  \yy(0) = \yy_0.
\end{equation}

In this study, we focus on \textit{fractional-step methods}, as termed
by Yanenko~\cite{Yanenko1971}, whereby the various terms $\Fl[\ell]$
of the right-hand side of the ODE are integrated in turn. The output
from a given sub-integration is then used as input to the next one.  An
approximation to the solution $\yy(t)$ is eventually produced when all
the terms have been appropriately integrated.

Fractional-step methods are based on two fundamental parts: the
(operator) splitting method and the sub-integrators. The simplest and
most well-known examples of operator-splitting methods for ODEs
include the Lie--Trotter~\cite{Trotter1958} or
Godunov~\cite{Godunov1959} splitting method and the Strang--Marchuk
splitting method~\cite{Strang1968,Marchuk1971}. These are low-order
methods (first and second order, respectively). Symmetrized methods,
whereby one splitting method is applied in tandem with its adjoint,
are a popular approach for achieving higher-order splitting methods;
the Strang--Marchuk splitting method can be derived from the
Lie--Trotter/Godunov method in this fashion~\cite{McLachlan2002}.  The
sub-integrators can range anywhere from an exact sub-flow to a
standard numerical method such as linear multistep or Runge--Kutta.
The classical order of the overall fractional-step method is then
generally the minimum of the order of the splitting method and all the
sub-integrators.  The class of multi-rate methods uses sub-stepping,
perhaps in an adaptive fashion using a separate integrator library
like SUNDIALS or MATLAB's \texttt{ode15s}, to perform sub-integration
to within a specified error tolerance~\cite{Ropp2004}.

When Runge--Kutta methods are used as the sub-integrators for each
fractional step, the result is a fractional step Runge--Kutta (FSRK)
method that can be cast in the framework of a generalized-structure
additive Runge--Kutta (GARK) method. Such representations have
appeared to various degrees of generality,
e.g.,~\cite{Christlieb2015,Gonzalez-Pinto2021}. Here, we show how to
systematically construct the Butcher tableau representation of a
general FSRK method, i.e., one having an arbitrary Runge--Kutta method
as sub-integrator at each of $s$ splitting stages and $N$ operators.

Linear stability is an important property of a numerical method.  It
is generally an important indicator in the design and performance of a
numerical method in practice. In this paper, we use the Butcher
tableau representation of GARK methods to study the linear stability
of FSRK methods; we give an interpretation of the stability
function in terms of the splitting method coefficients and the
individual Runge--Kutta methods; and we show how the linear
stability theory presented to explain common observations in
published studies on the stability behavior of
fractional-step methods.

Before proceeding further, it should be noted that it is widely
accepted that no single numerical method is a silver bullet that will
outperform all other methods on all problems. Fractional-step methods
are no exception. Well-known issues with the use of splitting methods
in general include the specification of boundary
conditions~\cite{Hundsdorfer2003} as well as convergence to spurious
steady states~\cite{Speth2013,Glowinski2017}. A significant body of
literature exists to address these and other issues surrounding the
implementation of operator-splitting methods in practice, but a full
discussion is beyond the scope of this study.

The remainder of the paper is organized as follows. The necessary
definitions and theoretical background on operator-splitting, GARK,
and FSRK methods are given in \cref{sec:background}.  The main
theoretical results on the Butcher tableau representation and linear
stability of FSRK methods appear in \cref{sec:main}.  Some examples on
the use of these theoretical results are provided in
\cref{sec:numerical_ex}. The examples illustrate how observations in
the literature can be explained in the general framework set out in
this paper.  Conclusions follow in \cref{sec:conclusions}.

\section{Background}
\label{sec:background}

In this section, we present some necessary background to construct
FSRK methods, including the definition of operator-splitting methods,
additive Runge--Kutta (ARK) methods as introduced in \cite{Cooper1980},
and their evolution to GARK methods presented in \cite{sandu2015}.

\subsection{Operator-splitting methods}

We begin by presenting operator-split\-ting me\-thods as discussed in
\cite{hairer2006}. Let $\pphi{\ell}{\Dt}$ be the flow of the subsystem
$\displaystyle \dv{\yyy[\ell]}{t} = \Fl[\ell](t,\yyy[\ell])$ for
$\ell = 1,2,\dots,\Nop$. Compositions of $\pphi{\ell}{\Dt}$ for
$\ell=1,2\dots,\Nop$, such as
\begin{subequations}
\begin{align}
	& \Phi_{\Dt} :=  \pphi{\Nop}{\Dt} \circ \pphi{\Nop-1}{\Dt}\circ \cdots \circ \pphi{1}{\Dt}, \label{eq:os_godunov} \\ 
	& \Phi_{\Dt}^\ast :=  \pphi{1}{\Dt} \circ \pphi{2}{\Dt}\circ \cdots \circ \pphi{\Nop}{\Dt}  \label{eq:os_godunovadj}
\end{align}
\end{subequations}
are two numerical methods to solve \cref{cauchy_problem}. The two
methods \cref{eq:os_godunov} and \eqref{eq:os_godunovadj} are adjoints
of each other and are both first-order accurate. In particular,
\cref{eq:os_godunov} is known as the Lie--Trotter (or Godunov)
splitting method, although the same name could apply to
\eqref{eq:os_godunovadj} by a re-numbering of the operators. The
second-order Strang--Marchuk splitting method can be viewed a
composition of the Lie--Trotter method and its adjoint with halved
step sizes and can be written as
\begin{align*}
\Psi_{\Dt}^{S} & = \Phi_{\Dt/2}^\ast \circ \Phi_{\Dt/2} \\
& = \pphi{1}{\Dt/2} \circ \pphi{2}{\Dt/2}\circ \cdots
\circ  \pphi{\Nop-1}{\Dt/2}\circ \pphi{\Nop}{\Dt} \circ \pphi{\Nop-1}{\Dt/2}\circ \cdots \circ \pphi{1}{\Dt/2}.
\end{align*}

\begin{rmk}
We note that the term $\pphi{\Nop}{\Dt}$ arises from the group
property of exact flows. This term is often approximated
directly. However, it is possible to approximate the two occurrences
$\pphi{\Nop}{\Dt/2}$ separately, leading to a \emph{different
	numerical method (with different accuracy and stability
	properties)}.
\end{rmk}

The general form of the operator-splitting method considered in this
paper is expressed as follows. Let
$\aalpha = \{ \aalpha_1, \aalpha_2,\dots,\aalpha_s \}$, where
$\aalpha_k = \{\aaalpha{k}{1}, \aaalpha{k}{2},\dots,\aaalpha{k}{\Nop}\}$,
$k=1,2,\ldots,s$, be the coefficients of the splitting method. An
$s$-stage operator-splitting method that solves \eqref{cauchy_problem}
can be written as
\begin{equation} \label{eq:os_method}
\Psi_{\Dt} := \prod_{k=1}^{s} \Phi_{\aalpha_k \Dt}^{\{k\}} = \Phi_{\aalpha_s \Dt}^{\{s\}} \circ \Phi_{\aalpha_{s-1} \Dt}^{\{s-1\}} \circ \cdots \circ \Phi_{\aalpha_1 \Dt}^{\{1\}},  
\end{equation}
where
$\Phi_{\aalpha_k \Dt}^{\{k\}} := \pphi{\Nop}{\aaalpha{k}{\Nop}\Dt}
\circ\pphi{\Nop-1}{\aaalpha{k}{\Nop-1}\Dt} \circ \cdots \circ
\pphi{1}{\aaalpha{k}{1}\Dt} $. The operator-splitting method
\cref{eq:os_method} can be viewed as a general additive method but
where only specific coupling between the operators is allowed (see
below). Hence, the accuracy and stability properties can be expected
to be inferior to additive methods. Additive methods, however, are not
applicable for co-simulations where the simulations of subsystems must
be treated as black boxes and data between subsystems can only be
exchanged after a subsystem is integrated. Hence, the study of
operator-splitting methods of the form~\cref{eq:os_method} have a
broad range of application despite their rather specific nature.

\subsection{RK and ARK methods}

\begin{mydef}(Runge--Kutta method) \label{def:rk} Let $b_i$ and
$a_{ij}, i,j=1,2, \dots,\ts$, be real numbers, and let
$c_i=\sum\limits_{j=1}^{\ts}a_{ij}$. One step of an $\ts$-stage
Runge--Kutta method is given by
\begin{subequations}\label{eq:rk}
	\begin{align}
		\mathbf{y}_{n+1} & = \mathbf{y}_n +\Delta t\sum_{i=1}^{\ts}b_i\F(t_n+c_i\Delta t, \mathbf{Y}_i), \\ 
		\mathbf{Y}_i & =\mathbf{y}_n +\Delta t \sum_{j=1}^{\ts} a_{ij}\F(t_n+c_j\Delta t, \mathbf{Y}_j),  \enskip i=1,2, \dots,\ts.
	\end{align}
\end{subequations}
\end{mydef}

The coefficients $b_i$, $c_i$, and $a_{ij}$, $i,j=1,2, \dots,\ts$, of a
Runge--Kutta method can be represented as the Butcher tableau
\begin{equation*}
\begin{array}{c|cccc}
	{c}_1&a_{11} &a_{12} & \dots &a_{1\ts}\\
	{c}_2&a_{21} &a_{22} & \dots &a_{2\ts}\\
	\vdots &  \vdots & \vdots & \vdots & \vdots \\
	{c}_{\ts}&a_{\ts 1} &a_{\ts 2} & \dots& a_{\ts \ts}\\
	\hline
	& {b}_1 & {b}_2 &  \dots& {b}_{\ts}\\
\end{array} = 
\begin{array}{c|c}
	\mathbf{c}&\mathbf{A}\\
	\hline
	& \mathbf{b}\\
\end{array}.
\end{equation*}
For notational simplicity, we denote the quadrature weights of the
Butcher tableau by $\mathbf{b}$ rather than $\mathbf{b}^T$.

When different $\ts$-stage Runge--Kutta integrators are applied to
each operator $\Fl[\ell]$ of~\cref{cauchy_problem}, the numerical
method is called an additive Runge--Kutta method \cite{Cooper1980,
kennedy2003}.

\begin{mydef}[Additive Runge--Kutta method]\label{def:ark}
Let $b_i^{[\ell]}$, $a_{ij}^{[\ell]}$,
$i,j=1,2,\dots,\ts$, $\ell= 1,2,\dots, \Nop$, be real numbers, and let
$c_j^{[\ell]}=\sum\limits_{i=1}^{\ts}a_{ij}^{[\ell]}$. One step of an
$\ts$-stage ARK method is given by
\begin{subequations} 
	\begin{align*}
		\mathbf{y}_{n+1} & = \mathbf{y}_n +\Delta t\sum_{\ell=1}^{\Nop} \sum_{i=1}^{ \ts}b_i^{[\ell]}\Fl[\ell](t_n+c_i^{[\ell]}\Delta t, \mathbf{Y}_i),\\
		\mathbf{Y}_i & = \mathbf{y}_n +\Delta t \sum_{\ell=1}^{\Nop}\sum_{j=1}^{\ts} a_{ij}^{[\ell]}\Fl[\ell](t_n+c_j^{[\ell]}\Delta t, \mathbf{Y}_j),\enskip i=1,2, \dots,\ts ,
	\end{align*}
\end{subequations}
where $b_i^{[\ell]}$, $c_j^{[\ell]}$, and $a_{ij}^{[\ell]}$ are
the coefficients of the method applied to operator $\Fl[\ell]$.
\end{mydef}
The Butcher tableau for ARK methods can be written as \cite{sandu2015}
\begin{equation}\label{eq:arktab}
\begin{array}{c|c|c|c|c|c|c|c|}
	\mathbf{c}^{[1]}&\mathbf{c}^{[2]}&\cdots&\mathbf{c}^{[\Nop]}&\mathbf{A}^{[1]}&\mathbf{A}^{[2]}&\cdots &\mathbf{A}^{[\Nop]}\\
	\hline
	&&&&\mathbf{b}^{[1]}&\mathbf{b}^{[2]}&\cdots&\mathbf{b}^{[\Nop]}\\
\end{array} ,
\end{equation}
where $\mathbf{A}^{[\ell]}$, $\mathbf{b}^{[\ell]}$,
$\mathbf{c}^{[\ell]}, \enskip \ell=1,2,\dots, \Nop$, are the
coefficients of the Runge--Kutta method associated with operator
$\ell$.

\subsection{GARK methods}
In \cite{sandu2015}, ARK methods were expanded to the family of
generalized additive Runge--Kutta (GARK) methods. For the
purposes of the FSRK methods considered in this paper, we define
GARK methods as follows.

\begin{mydef}[Generalized Additive Runge--Kutta (GARK)
	method] \label{def:gark} Let $b_j^{[\ell]}$ and
	$a_{ij}^{[\ell',\ell]}, i=1, 2,\dots,\ts^{[\ell']},\ j=1,2,\dots,
	\ts^{[\ell]},$ $\ell', \ell= 1,2,\dots, \Nop$, be
	real numbers, and let
	$c_i^{[\ell',\ell]}=\sum\limits_{j=1}^{\ts^{[\ell]}}a_{ij}^{[\ell',\ell]}$. One
	step of a GARK method with an $\Nop$-additive splitting of the
	right-hand side of \eqref{cauchy_problem} with $\Nop$
	stages reads
	\begin{subequations}\label{eq:gark}
		\begin{align}
			\mathbf{y}_{n+1} & = \mathbf{y}_n +\Delta t\sum_{\ell=1}^N\sum_{i=1}^{\ts^{[\ell]}} b_i^{[\ell]}\Fl[\ell](t_n+c_i^{[\ell,\ell]}\Delta t, \mathbf{Y}^{[\ell]}_i),  \\
			\mathbf{Y}^{[\ell']}_i & = \mathbf{y}_n +\Delta t
			\sum\limits_{\ell=1}^{\Nop}
			\sum_{j=1}^{\ts^{[\ell]}}
			a_{ij}^{[\ell',\ell]}\Fl[\ell](t_n+c_j^{[\ell',\ell]}\Delta
			t, \mathbf{Y}^{[\ell]}_j),
			\enskip i=1,2,
			\dots,\ts^{[\ell']},\ \ell'=1,2,\dots,N. 
		\end{align}
	\end{subequations}
	The corresponding generalized Butcher tableau is
	\begin{equation}\label{gark_tab}
		\begin{array}{cccc|cccc}
			\cc[1,1]  & \cc[1,2]    &  \cdots  & \cc[1,N]     & 
			\AAA[1,1] & 	\AAA[1,2] & \cdots & 	\AAA[1,N] \\
			\cc[2,1]  & \cc[2,2]    &  \cdots  & \cc[2,N]       & 
			\AAA[2,1] & 	\AAA[2,2] & \cdots & 	\AAA[2,N] \\
			\vdots  & \vdots & \ddots & \vdots    & 
			\vdots  & \vdots & \ddots & \vdots \\
			\cc[\Nop,1]  & \cc[\Nop,2]    &  \cdots  & \cc[\Nop,N]       & 
			\AAA[\Nop,1] & 	\AAA[\Nop,2] & \cdots & 	\AAA[\Nop,N] \\
			\hline 
			&    &    &     & 
			\bb[1]  & \bb[2] & \cdots & \bb[N] \\
		\end{array}
	\end{equation}
\end{mydef}

\begin{rmk}
GARK methods generalize the structure of ARK methods in the sense
that different operators of the right-hand side of
\cref{cauchy_problem} can be integrated by Runge--Kutta methods with
different numbers of stages. The diagonal matrix $\AAA[\ell,\ell]$
corresponds to the Runge--Kutta method used to integrate 
operator $\ell$. The off-diagonal terms $\AAA[\ell',\ell]$,
$\ell' \neq \ell$, represent the coupling between operators within a
stage.

In \cite{sandu2015}, \cref{def:gark} is generalized in the
	following two ways. First, the number of rows in \cref{gark_tab},
	representing the number of GARK stages, is $\Nop'$, which can be
	less than the number of operators, $\Nop$.  Second, the
	$\YY_i^{[\ell]}$ in the argument of the operators $\Fl[\ell]$ in
	\cref{eq:gark} can be generalized to $\YY_i^{[J(\ell)]}$, where
	the mapping $J: \{1,2,\dots,\Nop\}\rightarrow \{1,2,\dots,\Nop'\}$
	from the operators to the GARK stages may not be the
	identity. Details of these generalizations are given in
	\cite{sandu2015} but are beyond the scope of the analysis
	presented here.  

\end{rmk}

\begin{mydef}[internal consistency of GARK
methods] \label{def:int_consis_gark} A GARK method \cref{eq:gark} is
called internally consistent \cite{sandu2015} if
\begin{equation} \label{eq:int_consis_gark}
	\sum\limits_{j=1}^{\ts^{[1]}} a_{ij}^{[\ell',1]} = \cdots =
	\sum\limits_{j=1}^{\ts^{[\Nop]}} a_{ij}^{[\ell',\Nop]} =
	c_{i}^{[\ell',\ell']} , \enskip i = 1, 2,\dots, \ts^{[\ell']},
	\enskip \ell' = 1,2,\dots,\Nop.
\end{equation}
\end{mydef}

\begin{rmk}
The internal consistency condition \cref{eq:int_consis_gark} ensures
that all intermediate stages are computed at the same internal
times. If a method is internally consistent, the matrix
$[\cc[i,j]]$, $i, j=1,2,\dots,N$, in \cref{gark_tab}
can be represented as a single column
$[\cc[1,1], \cc[2,2],\dots,\cc[N,\Nop]]^T$.
\end{rmk}

As described in \cite{sandu2015}, any GARK method can be written as an
ARK method (and vice versa)
; i.e., the Butcher tableau of a GARK method \eqref{gark_tab}
can be written in the form an ARK method \eqref{eq:arktab} but with
more stages, and the Butcher tableau of an ARK method
\eqref{eq:arktab} is a special case of that of a GARK method
\eqref{gark_tab} with one stage.

\subsection{FSRK methods}

When solving \eqref{cauchy_problem}, one can use operator-splitting
methods combined with suitable Runge--Kutta methods to integrate each
operator. We call this an FSRK method as defined below.

\begin{mydef}[FSRK method] \label{def:FSRK}
Consider \eqref{cauchy_problem}, and assume that we advance the time
integration by choosing a combination of an $s$-stage OS method and
Runge--Kutta time-stepping methods. Let
$\{ \alpha_k^{[\ell]} \}_{k=1,2,\dots,s}^{\ell=1,2,\dots,N}$ be the
coefficients of the OS method. Let $\begin{array}{c|c}
	\tcc[\ell]_k & \tA[\ell]_k \\
	\hline
	& \tbb[\ell]_k \\
\end{array}$ be the Butcher tableau of the $\ts^{[\ell]}_k$-stage Runge--Kutta method applied to operator $l$ at OS stage $k$. 
Then, one step of an FSRK method reads

	\begin{subequations}\label{eq:fsrk}
		\begin{align}\label{eq:fsrk1}
			& \yy_{n+1} = \yy_n 
			+  \Dt \sum\limits_{k=1}^{s}\sum\limits_{\ell=1}^{\Nop} \sum\limits_{i=1}^{\ts_{k}^{[\ell]}} \aaalpha{k}{\ell} \tb[\ell]_{k,i} \Fl[\ell] \left(t_k^{[\ell]} + \tc[\ell]_{k,i}\aaalpha{k}{\ell}\Dt, \YY_{k,i}^{[\ell]}\right), \\
			\label{eq:fsrk2}
			& \begin{aligned}
				\YY_{k,i}^{[\ell]} = \yy_n  & + \Dt \sum\limits_{k'=1}^{k-1}\sum\limits_{\ell'=1}^{\Nop} \sum\limits_{i=1}^{\ts_{k'}^{[\ell']}} \aaalpha{k'}{\ell'} \tb[\ell']_{k',i} \Fl[\ell'] \left(t_{k'}^{[\ell']} + \tc[\ell']_{k',i}\aaalpha{k'}{\ell'}\Dt, \YY_{k',i}^{[\ell']}\right)   \\
				& + 
				\Dt \sum\limits_{\ell'=1}^{\ell-1}\sum\limits_{i=1}^{\ts_{k}^{[\ell']}}\aaalpha{k}{\ell'} \tb[\ell']_{k,i} \Fl[\ell'] \left(t_k^{[\ell']} + \tc[\ell']_{k,i}\aaalpha{k}{\ell'}\Dt, \YY_{k,i}^{[\ell']}\right)   \\
				& + 
				\Dt  \sum\limits_{j=1}^{\ts_{k}^{[\ell]}} \aaalpha{k}{\ell} \ta[\ell]_{k,ij} \Fl[\ell] \left(t_k^{[\ell]} + \tc[\ell]_{k,j}\aaalpha{k}{\ell}\Dt, \YY_{k,j}^{[\ell]}\right) ,
			\end{aligned}
		\end{align}
	\end{subequations}
	where $\ta[\ell]_{k,ij}$ is entry $(i,j)$ of $ \tA[\ell]_k$,
	$\tb[\ell]_i$ and $\tc[\ell]_i$ are entry $i$ of $\tbb[\ell]$ and
	$\tcc[\ell]$ respectively, and
	$t_k^{[\ell]} =t_n+\sum\limits_{k'=1}^{k-1} \aaalpha{k'}{\ell}\Dt$
	is the time for operator $\ell$ at the beginning of 
	operator-splitting stage $k$.  
\end{mydef}

\section{Main results}
\label{sec:main}
In this section, we solve \eqref{cauchy_problem} using the
operator-splitting method \eqref{eq:os_method} where each subsystem
$\displaystyle \dv{\yyy[\ell]}{t} = \Fl[\ell](t,\yyy[\ell])$ is
integrated using a Runge--Kutta method \eqref{eq:rk}. Because a
Runge--Kutta method is applied to a subsystem that is usually solved
over a fraction, $\alpha_k^{[\ell]}$, of $\Dt$, this results in an
FSRK method. We show that the FSRK can be regarded as a GARK method,
present the Butcher tableau associated with it, and analyze its
stability.

We first construct the Butcher tableau of an FSRK method in
\cref{th:butcher_theorem}.

\begin{thm}\label{th:butcher_theorem}
	The FSRK method~\cref{def:FSRK} applied to~\eqref{cauchy_problem}
	can be represented as an extended Butcher tableau with the 
	structure~\eqref{eq:arktab} that incorporates the coefficients of
	the Runge--Kutta integrators scaled by the coefficients of the OS
	method. 
	The entries of \cref{eq:arktab} take the form
\begin{equation}\label{eq:Abcell}
	\begin{aligned}
		& \AAA[\ell]  = 
		\begin{bmatrix}
			\AAA[\ell]_1  & & & & \\
			\ee \bb[\ell]_1   & \AAA[\ell]_2  & &  & \\
			\vdots   & \ee\bb[\ell]_2  & \ddots &  & \\
			\vdots	  &  \vdots	& & \ddots &  \\
			\ee \bb[\ell]_1   & \ee \bb[\ell]_2 & \cdots  & \ee \bb[\ell]_{s-1} & \AAA[\ell]_s \\
		\end{bmatrix},  \\
		& \bb[\ell] = \begin{bmatrix}
			\bb[\ell]_1 & \bb[\ell]_2 & \cdots & \bb[\ell]_s 
		\end{bmatrix}, \\
		& \cc[\ell] = \begin{bmatrix}
				\cc[\ell]_1 \\
				\cc[\ell]_2 \\
				\vdots \\
				\cc[\ell]_s 
		\end{bmatrix}, \qquad \ell=1,2,\dots,\Nop, 
	\end{aligned}
\end{equation}
where each matrix $\AAA[\ell]$ is a block lower-triangular matrix of
size $\bbS{}{}\times \bbS{}{}$ and each row vector $\bb[\ell]$ is a
block vector of size $1 \times \bbS{}{}$, where
$\bbS{}{}
= \sum\limits_{k=1}^s \sum\limits_{\ell=1}^\Nop \ts_{k}^{[\ell]}$, and
$\ee$ denotes a column vector of ones.  Diagonal block $k$ of
$\AAA[\ell]$ is denoted by $\AAA[\ell]_k$ of size
$\bbS{k}{\Nop} \times \bbS{k}{\Nop}$, where
$\bbS{k}{\ell}= \sum\limits_{i=1}^\ell \tilde{s}_k^{[i]}$,
block $k$ of $\bb[\ell]$ is denoted by $\bb[\ell]_k$ of size
$1\times \bbS{k}{\Nop}$, block $k$ of $\cc[\ell]$ is denoted by
$\cc[\ell]_k$ of size $\bbS{k}{\Nop} \times 1$ with
\begin{equation}\label{eq:Abc_k_ell}
	\begin{aligned}
		& \AAA[\ell]_k  = 
		\begin{bmatrix}
			\Zero_{\bbS{k}{\ell-1} \times \bbS{k}{\ell-1} } & \Zero_{\bbS{k}{\ell-1} \times \ts_{k}^{[\ell]}  }  & \Zero_{\bbS{k}{\ell-1} \times (\bbS{k}{\Nop} -\bbS{k}{\ell}) } \\[2ex]
			\Zero_{\ts_{k}^{[\ell]} \times \bbS{k}{\ell-1}} & \alpha_k^{[\ell]} \tA[\ell]_k & \Zero_{\ts_{k}^{[\ell]} \times (\bbS{k}{\Nop}-\bbS{k}{\ell})} \\[2ex]
			\Zero_{(\bbS{k}{\Nop}-\bbS{k}{\ell}) \times \bbS{k}{\ell-1}} & \alpha_k^{[\ell]} \ee \tbb[\ell]_k  & \Zero_{(\bbS{k}{\Nop}-\bbS{k}{\ell}) \times (\bbS{k}{\Nop}-\bbS{k}{\ell})} \\ 
		\end{bmatrix}, \\
		& \bb[\ell]_k = 
		\begin{bmatrix}
			\Zero_{1\times \bbS{k}{\ell -1}} & \alpha_k^{[\ell]} \tbb[\ell]_k & \Zero_{1\times (\bbS{k}{\Nop} - \bbS{k}{\ell})}.
		\end{bmatrix}, \\
		& \cc[\ell]_{k} = \begin{bmatrix}
			(\sum\limits_{i=1}^{k-1} \alpha_i^{[\ell]})\ee_{\bbS{k}{\ell-1}} \\
			(\sum\limits_{i=1}^{k-1} \alpha_i^{[\ell]}) \ee_{\ts_{k}^{[\ell]}} + \alpha_{k}^{[\ell]} \tcc[\ell]_k \\
			(\sum\limits_{i=1}^{k} \alpha_i^{[\ell]}) \ee_{\bbS{k}{\Nop} -\bbS{k}{\ell}}
		\end{bmatrix}.
	\end{aligned}
\end{equation}

\end{thm}

\begin{proof}

	When solving \cref{cauchy_problem} using the FSRK
	method~\cref{eq:fsrk}, let $\YY^{[\ell]}_k$ be the intermediate solution
	values after solving operator $\ell$ at operator-splitting
	stage $k$. Let
	$\displaystyle \{\YY_{k,j}^{[\ell]} \}_{j=1}^{\ts_{k}^{[\ell]}}$
	be the intermediate Runge--Kutta solution values when solving operator
	$\ell$ at operator-splitting stage $k$. Let
	$t_k^{[\ell]} =t_n+\sum\limits_{k'=1}^{k-1} \aaalpha{k'}{\ell}\Dt$
	be the time for operator $\ell$ at the beginning of
	operator-splitting stage $k$.
	\\
	When finding $\YY_k^{[\ell]}$, we apply the Runge--Kutta method
	with Butcher tableau $\begin{array}{c|c}
		\tcc[\ell]_k & \tA[\ell]_k \\
		\hline
		& \tbb[\ell]_k \\
	\end{array}$ to operator $\ell$ 
	with initial condition $\YY_{k,0}^{[\ell]}$ over the interval $[t_k^{[\ell]},t_k^{[\ell]}+ \alpha_k^{[\ell]}\Dt]$. The initial condition $\YY_{k,0}^{[\ell]}$ is defined as the following piecewise function: 
	\begin{equation} \label{FSRK_ic}
		\YY_{k,0}^{[\ell]} = \left\{ \begin{array}{ll}
			\yy_n, & \text{ if } \ell =1, k=1, \\[1ex]
			\YY_{k-1}^{[\Nop]}, & \text{ if } \ell =1, k>1, \\[1ex]
			\YY_{k}^{[\ell-1]}, & \text{ if } \ell > 1. \\[1ex]
		\end{array}\right. 
	\end{equation}
	Applying one-step of Runge--Kutta method to find $\YY_k^{[\ell]}$, we
	get
	\begin{subequations}\label{eq:FSRK_recursive}
		\begin{align}\label{eq:FSRK_recursive_Yk}
			\YY^{[\ell]}_k & = \YY_{k,0}^{[\ell]} +  \aaalpha{k}{\ell} \Dt  \sum\limits_{i=1}^{\ts_{k}^{[\ell]}}  \tb[\ell]_{k,i} \Fl[\ell] (t_k^{[\ell]}  + \tc[\ell]_{k,i}\aaalpha{k}{\ell}\Dt, \YY_{k,i}^{[\ell]}), \\
			\label{eq:FSRK_recursive_Ykj}
			\YY_{k,i}^{[\ell]} & = \YY_{k,0}^{[\ell]} + 
			\aaalpha{k}{\ell} \Dt  \sum\limits_{j=1}^{\ts_{k}^{[\ell]}}  \ta[\ell]_{k,ij} \Fl[\ell] (t_k^{[\ell]} + \tc[\ell]_{k,j}\aaalpha{k}{\ell}\Dt, \YY_{k,j}^{[\ell]}).
		\end{align}
	\end{subequations}
	Using the recursive definition \cref{eq:FSRK_recursive_Yk}, we can
	find a general formula for $\YY_{k}^{[\ell]}:$
	\begin{equation}\label{eq:FSRK_general_Yk}
		\begin{aligned}
			\YY^{[\ell]}_k = \yy_n & +\Dt \sum\limits_{k'=1}^{k-1}\sum\limits_{\ell'=1}^{\Nop} \sum\limits_{i=1}^{\ts_{k'}^{[\ell']}} \aaalpha{k'}{\ell'} \tb[\ell']_{k',i} \Fl[\ell'] (t_{k'}^{[\ell']} + \tc[\ell']_{k',i}\aaalpha{k'}{\ell'}\Dt, \YY_{k',i}^{[\ell']})    \\
			& + \Dt \sum\limits_{\ell'=1}^{\ell} \sum\limits_{i=1}^{\ts_{k}^{[\ell']}} \aaalpha{k}{\ell'} \tb[\ell']_{k,i} \Fl[\ell'] (t_k^{[\ell']} + \tc[\ell']_{k,i}\aaalpha{k}{\ell'}\Dt, \YY_{k,i}^{[\ell']}).
		\end{aligned}
	\end{equation}
	To find an explicit formula for the initial condition
	$\YY_{k,0}^{[\ell]}$ in \cref{eq:FSRK_recursive}, we substitute
	$\YY_{k-1}^{[\Nop]}$ and $\YY_{k}^{[\ell-1]}$
	using~\cref{eq:FSRK_general_Yk} into \cref{FSRK_ic} to yield
	\begin{equation}\label{FSRK_ic_explicit}
		\YY_{k,0}^{[\ell]} = \left\{ \begin{array}{ll}
			\yy_n, &  \text{ if } \ell =1, k=1, \\[1ex]
			\begin{aligned}
				\yy_n  & +\Dt \sum\limits_{k'=1}^{k-1}\sum\limits_{\ell'=1}^{\Nop} \sum\limits_{i=1}^{\ts_{k'}^{[\ell']}} \aaalpha{k'}{\ell'} \tb[\ell']_{k',i} \Fl[\ell'] (t_{k'}^{[\ell']} + \tc[\ell']_{k',i}\aaalpha{k'}{\ell'}\Dt, \YY_{k',i}^{[\ell']})   
			\end{aligned}, & \text{ if } \ell =1, k>1, \\[1ex]
			\begin{aligned}
				\yy_n  & +\Dt \sum\limits_{k'=1}^{k-1}\sum\limits_{\ell'=1}^{\Nop} \sum\limits_{i=1}^{\ts_{k'}^{[\ell']}} \aaalpha{k'}{\ell'} \tb[\ell']_{k',i} \Fl[\ell'] (t_{k'}^{[\ell']} + \tc[\ell']_{k',i}\aaalpha{k'}{\ell'}\Dt, \YY_{k',i}^{[\ell']})    \\
				& + \Dt \sum\limits_{\ell'=1}^{\ell-1} \sum\limits_{i=1}^{\ts_{k}^{[\ell']}} \aaalpha{k}{\ell'} \tb[\ell']_{k,i} \Fl[\ell'] (t_k^{[\ell']} + \tc[\ell']_{k,i}\aaalpha{k}{\ell'}\Dt, \YY_{k,i}^{[\ell']})
			\end{aligned} , & \text{ if } \ell > 1. \\[1ex]
		\end{array}\right. 
	\end{equation}
	Substituting \cref{FSRK_ic_explicit} into
	\cref{eq:FSRK_recursive_Ykj}, we recover \cref{eq:fsrk2}.  To
	construct a Butcher tableau that includes all the data, we need a
	tableau of size $\bbS{}{} \times \bbS{}{}$, consisting of $sN$ blocks
	of sizes $\ts_{k}^{[\ell]}$. Each block corresponds to
	$\{\YY_{k,j}^{[\ell]}\}_{j=1}^{\ts_{k}^{[\ell]}}$. For clarity, we
	mark those $\{\YY_{k,j}^{[\ell]}\}_{j=1}^{\ts_{k}^{[\ell]}}$ on
	the tableau \cref{tab:comp_tableau_proof}. \Cref{eq:fsrk2} implies
	that in the row block corresponds to
	$\{\YY_{k,j}^{[\ell]}\}_{j=1}^{\ts_{k}^{[\ell]}}$, the $(k',\ell')$
	block entry is
	$$\left\{
	\begin{array}{ll}
		\alpha_{k'}^{[\ell']} \ee \tbb[\ell']_{k'}, & \text{ if } k'<k \text{ and } \ell'\leq \Nop, \\[1ex]
		\alpha_{k'}^{[\ell']} \ee \tbb[\ell']_{k'}, & \text{ if } k=k \text{ and } \ell' < \ell, \\[1ex]
		\alpha_{k'}^{[\ell']} \tA[\ell']_{k'}, & \text{ if } k=k \text{ and } \ell' = \ell, \\[1ex]
		\Zero, & \text{ if } k'>k.
	\end{array}\right. $$ 
	An example of a row block that corresponds to
	$\{\YY_{k,j}^{[\ell]}\}_{j=1}^{\ts_{k}^{[\ell]}}$ is shown in
	\cref{tab:comp_tableau_proof}.
	\begin{equation} \label{tab:comp_tableau_proof}
		\resizebox{0.9\textwidth}{!}{
			$\begin{array}{c|cccc|c|ccccc|c|cccc}
				&\{\YY_{1,j}^{[1]}\}_{j=1}^{\ts_{1}^{[1]}} & \{\YY_{1,j}^{[2]}\}_{j=1}^{\ts_{1}^{[2]}} & 
				\cdots &
				\{\YY_{1,j}^{[\Nop]}\}_{j=1}^{\ts_{1}^{[\Nop]}} &
				\cdots  & 
				\{\YY_{k,j}^{[1]}\}_{j=1}^{\ts_{k}^{[1]}} & \cdots & \{\YY_{k,j}^{[\ell]}\}_{j=1}^{\ts_{k}^{[\ell]}} & 
				\cdots &
				\{\YY_{k,j}^{[\Nop]}\}_{j=1}^{\ts_{k}^{[\Nop]}} &
				\cdots  & 
				\{\YY_{s,j}^{[1]}\}_{j=1}^{\ts_{s}^{[1]}} & \{\YY_{s,j}^{[2]}\}_{j=1}^{\ts_{s}^{[2]}} & 
				\cdots &
				\{\YY_{s,j}^{[\Nop]}\}_{j=1}^{\ts_{s}^{[\Nop]}}  \\
				\hline 
				\{\YY_{1,j}^{[1]}\}_{j=1}^{\ts_{1}^{[1]}} &   
				&&&&&   & 
				&&&&   & 
				&&& \\
				\{\YY_{1,j}^{[2]}\}_{j=1}^{\ts_{1}^{[2]}} &   
				&&&&&   & 
				&&&&   & 
				&&& \\
				\vdots &   
				&&&&&   & 
				&&&&   & 
				&&& \\
				\{\YY_{1,j}^{[\Nop]}\}_{j=1}^{\ts_{1}^{[\Nop]}} &   
				&&&&&   & 
				&&&&   & 
				&&& \\ 
				\hline 
				\vdots &   
				&&&&&   & 
				&&&&   & 
				&&& \\
				\hline 
				\{\YY_{k,j}^{[1]}\}_{j=1}^{\ts_{k}^{[1]}} &   
				&&&&&   & 
				&&&&   & 
				&&& \\
				\vdots &   
				&&&&&   & 
				&&&&   & 
				&&& \\
				\{\YY_{k,j}^{[\ell]}\}_{j=1}^{\ts_{k}^{[\ell]}} & \aaalpha{1}{1} \ee \tbb[1]_1
				& \aaalpha{1}{2}\ee\tbb[2]_1& \cdots &\aaalpha{1}{\Nop}\ee\tbb[\Nop]_1& \cdots   & 
				\aaalpha{k}{1}\ee\tbb[1]_k 
				& \cdots 
				&\aaalpha{\ell}{k}\tA[\ell]_k & \Zero& \Zero &  \cdots  & 
				\Zero& \Zero & \cdots& \Zero \\
				\vdots &   
				&&&&&   & 
				&&&&   & 
				&&& \\
				\{\YY_{k,j}^{[\Nop]}\}_{j=1}^{\ts_{k}^{[\Nop]}} &   
				&&&&&   & 
				&&&&   & 
				&&& \\ 
				\hline 
				\vdots &   
				&&&&&   & 
				&&&&   & 
				&&& \\
				\hline 
				\{\YY_{s,j}^{[1]}\}_{j=1}^{\ts_{s}^{[1]}} &   
				&&&&&   & 
				&&&&   & 
				&&& \\
				\{\YY_{s,j}^{[2]}\}_{j=1}^{\ts_{s}^{[2]}} &   
				&&&&&   & 
				&&&&   & 
				&&& \\
				\vdots &   
				&&&&&   & 
				&&&&   & 
				&&& \\
				\{\YY_{s,j}^{[\Nop]}\}_{j=1}^{\ts_{s}^{[\Nop]}} &   
				&&&&&   & 
				&&&&   & 
				&&& \\ 
				\hline 
			\end{array}. $}
	\end{equation}
	Completing the remaining entries
	of~\cref{tab:comp_tableau_proof}, we obtain the compact
	tableau \cref{tab:extended_tab_comp} associated to the FSRK
	method. We note that the FSRK tableau
	\cref{tab:extended_tab_comp} is organized by
	operator-splitting stages. Keeping only block column $\ell$ in
	each stage block $k$ and filling the other entries with
	$\Zero$ leads to the form of $\AAA[\ell]$ in \cref{eq:Abcell}
	with the ARK structure. The values of $\bb[\ell]$ and
	$\cc[\ell]$ follow directly from \cref{eq:fsrk1}. 
\end{proof}

\begin{rmk}
An example of the extended Butcher tableau of ARK form \cref{eq:arktab} is given in
\cref{tab:extended_tab}. Essentially, in each block
$\AAA[\ell]_k$, the zeros are padding for operators other than
operator
$\ell$ that is being integrated. A more compact form of the Butcher
tableau that removes the zero padding and combines the
$\AAA[\ell]$ to reveal the block lower-triangular structure is given
in \cref{tab:extended_tab_comp}. 
Each diagonal
block of \cref{tab:extended_tab_comp} is a block lower-triangular
matrix that shows the specific structured coupling between the
operators of an FSRK method. As expected, the coupling between
operators is more restrictive than a general GARK method.  We note
that many published FSRK methods are not internally consistent. Even
if after each stage all the operators have the same abscissae,
internal consistency may fail at the stages of the Runge--Kutta
sub-integrators. Besides the ostensible drawback of not
	being able to interpret the stage values as the solution at a
	given time, the lack of internal consistency typically makes it
	more difficult to construct higher-order GARK methods due to the
	increased number of order conditions not automatically satisfied.
\end{rmk}

\begin{equation} \label{tab:extended_tab_comp}
\resizebox{0.925\textwidth}{!}{%
	$	\begin{array}{cccc|cccc|cccc|cccc}
		\alpha_{1}^{[1]} \tA[1]_1 & 0 & \cdots & 0  &
		&   &   &    &
		&   &   &   
		&   &   &\\
		\alpha_{1}^{[1]}\ee \tbb[1]_1 & \alpha_{1}^{[2]} \tA[2]_1 & 
		\ddots & 0  &
		&   &   &    &
		&   &   &   
		&   &   &\\
		\alpha_{1}^{[1]}\ee \tbb[1]_1 & \alpha_{1}^{[2]}\ee \tbb[2]_1 & \ddots & 0  &
		&   &   &   & 
		&   &   &  
		&   &   &\\
		\alpha_{1}^{[1]}\ee \tbb[1]_1 & \alpha_{1}^{[2]}\ee \tbb[2]_1 & \cdots & \alpha_{1}^{[\Nop]} \tA[\Nop]_1 &
		&   &   &    &
		&   &   &   
		&   &   &\\ 
		\hline 
		\alpha_{1}^{[1]}\ee \tbb[1]_1 & \alpha_{1}^{[2]}\ee \tbb[2]_1 & \cdots & \alpha_{1}^{[\Nop]}\ee \tbb[\Nop]_1 &
		\alpha_{2}^{[1]} \tA[1]_2 & 0  & \cdots  &  0 &   
		&   &   & 
		&   &   &\\ 
		\alpha_{1}^{[1]}\ee \tbb[1]_1 & \alpha_{1}^{[2]}\ee \tbb[2]_1 & \cdots & \alpha_{1}^{[\Nop]}\ee \tbb[\Nop]_1 &
		\alpha_{2}^{[1]} \ee \tbb[1]_2 & \alpha_{2}^{[2]} \tA[2]_2  & \ddots  &  0 &   
		&   &   & 
		&   &   &\\ 
		\alpha_{1}^{[1]}\ee \tbb[1]_1 & \alpha_{1}^{[2]}\ee \tbb[2]_1 & \cdots & \alpha_{1}^{[\Nop]}\ee \tbb[\Nop]_1 &
		\alpha_{2}^{[1]} \ee \tbb[1]_2 & \alpha_{2}^{[2]}\ee \tbb[2]_2  & \ddots  &  0 &   
		&   &   & 
		&   &   &\\ 
		\alpha_{1}^{[1]}\ee \tbb[1]_1 & \alpha_{1}^{[2]}\ee \tbb[2]_1 & \cdots & \alpha_{1}^{[\Nop]}\ee \tbb[\Nop]_1 &
		\alpha_{2}^{[1]} \ee \tbb[1]_2 & \alpha_{2}^{[2]}\ee \tbb[2]_2  & \cdots  & \alpha_{2}^{[\Nop]} \tA[\Nop]_2   &   
		&   &   & 
		&   &   &\\ 
		\hline 
		\alpha_{1}^{[1]}\ee \tbb[1]_1 & \alpha_{1}^{[2]}\ee \tbb[2]_1 & \cdots & \alpha_{1}^{[\Nop]}\ee \tbb[\Nop]_1 &
		\alpha_{2}^{[1]} \ee \tbb[1]_2 & \alpha_{2}^{[2]}\ee \tbb[2]_2  & \cdots  & \alpha_{2}^{[\Nop]} \ee \tbb[\Nop]_2   &   
		\ddots &   &   & 
		&   &   &\\ 
		\alpha_{1}^{[1]}\ee \tbb[1]_1 & \alpha_{1}^{[2]}\ee \tbb[2]_1 & \cdots & \alpha_{1}^{[\Nop]}\ee \tbb[\Nop]_1 &
		\alpha_{2}^{[1]} \ee \tbb[1]_2 & \alpha_{2}^{[2]}\ee \tbb[2]_2  & \cdots  & \alpha_{2}^{[\Nop]} \ee \tbb[\Nop]_2   &   
		& \ddots   &   & 
		&   &   &\\ 
		\alpha_{1}^{[1]}\ee \tbb[1]_1 & \alpha_{1}^{[2]}\ee \tbb[2]_1 & \cdots & \alpha_{1}^{[\Nop]}\ee \tbb[\Nop]_1 &
		\alpha_{2}^{[1]} \ee \tbb[1]_2 & \alpha_{2}^{[2]}\ee \tbb[2]_2  & \cdots  & \alpha_{2}^{[\Nop]} \ee \tbb[\Nop]_2   &   
		&   & \ddots   & 
		&   &   &\\ 
		\alpha_{1}^{[1]}\ee \tbb[1]_1 & \alpha_{1}^{[2]}\ee \tbb[2]_1 & \cdots & \alpha_{1}^{[\Nop]}\ee \tbb[\Nop]_1 &
		\alpha_{2}^{[1]} \ee \tbb[1]_2 & \alpha_{2}^{[2]}\ee \tbb[2]_2  & \cdots  & \alpha_{2}^{[\Nop]} \ee \tbb[\Nop]_2   &   
		&   &  & \ddots   & 
		&   &   &\\
		\hline 
		\alpha_{1}^{[1]}\ee \tbb[1]_1 & \alpha_{1}^{[2]}\ee \tbb[2]_1 & \cdots & \alpha_{1}^{[\Nop]}\ee \tbb[\Nop]_1 &
		\alpha_{2}^{[1]} \ee \tbb[1]_2 & \alpha_{2}^{[2]}\ee \tbb[2]_2  & \cdots  & \alpha_{2}^{[\Nop]} \ee \tbb[\Nop]_2   &   
		&   &    &  & 
		\alpha_{s}^{[1]} \tA[1]_s & 0  & \cdots  & 0 \\
		\alpha_{1}^{[1]}\ee \tbb[1]_1 & \alpha_{1}^{[2]}\ee \tbb[2]_1 & \cdots & \alpha_{1}^{[\Nop]}\ee \tbb[\Nop]_1 &
		\alpha_{2}^{[1]} \ee \tbb[1]_2 & \alpha_{2}^{[2]}\ee \tbb[2]_2  & \cdots  & \alpha_{2}^{[\Nop]} \ee \tbb[\Nop]_2   &   
		&   &    & & 
		\alpha_{s}^{[1]}\ee \tbb[1]_s & \alpha_{s}^{[2]} \tA[2]_s  & \ddots  & 0 \\
		\alpha_{1}^{[1]}\ee \tbb[1]_1 & \alpha_{1}^{[2]}\ee \tbb[2]_1 & \cdots & \alpha_{1}^{[\Nop]}\ee \tbb[\Nop]_1 &
		\alpha_{2}^{[1]} \ee \tbb[1]_2 & \alpha_{2}^{[2]}\ee \tbb[2]_2  & \cdots  & \alpha_{2}^{[\Nop]} \ee \tbb[\Nop]_2   &   
		&   &    & & 
		\alpha_{s}^{[1]}\ee \tbb[1]_s & \alpha_{s}^{[2]} \ee\tbb[2]_s  & \ddots & 0 \\
		\alpha_{1}^{[1]}\ee \tbb[1]_1 & \alpha_{1}^{[2]}\ee \tbb[2]_1 & \cdots & \alpha_{1}^{[\Nop]}\ee \tbb[\Nop]_1 &
		\alpha_{2}^{[1]} \ee \tbb[1]_2 & \alpha_{2}^{[2]}\ee \tbb[2]_2  & \cdots  & \alpha_{2}^{[\Nop]} \ee \tbb[\Nop]_2   &   
		&   &    & & 
		\alpha_{s}^{[1]}\ee \tbb[1]_s & \alpha_{s}^{[2]} \ee\tbb[2]_s  & \cdots & \alpha_{s}^{[\Nop]}  \tA[\Nop]_s\\
		\hline 
		\alpha_{1}^{[1]} \tbb[1]_1 & \alpha_{1}^{[2]} \tbb[2]_1 & \cdots & \alpha_{1}^{[\Nop]} \tbb[\Nop]_1 &
		\alpha_{2}^{[1]}  \tbb[1]_2 & \alpha_{2}^{[2]} \tbb[2]_2  & \cdots  & \alpha_{2}^{[\Nop]} \tbb[\Nop]_2   &   
		\cdots & \cdots  & \cdots   & \cdots & 
		\alpha_{s}^{[1]} \tbb[1]_s & \alpha_{s}^{[2]} \tbb[2]_s  & \cdots & \alpha_{s}^{[\Nop]}  \tbb[\Nop]_s\\
	\end{array}$
}
\end{equation}

\begin{equation} \label{tab:extended_tab}	
\resizebox{0.925\textwidth}{!}{
	$\begin{array}{c:c:c||c:c:c||c||c:c:c}
		\multicolumn{3}{c||}{\AAA[1]}  & \multicolumn{3}{c||}{\AAA[2]}  & \cdots  & \multicolumn{3}{c}{\AAA[\Nop]}  \\ 
		\hline 
		\begin{array}{ccc}
			\alpha_1^{[1]} \tA[1]_1 & 0 & 0   \\
			\alpha_1^{[1]} \ee \tbb[1]_1 & 0 & 0   \\
			\vdots  & \vdots  & \vdots   \\
			\vdots  & \vdots  & \vdots   \\
		\end{array}	&    &     &
		\begin{array}{ccc}
			0 & 0 & 0 \\
			0 & \alpha_1^{[2]} \tA[2]_1 & 0 \\
			0 & \alpha_1^{[2]} \ee\tbb[2]_1 & 0 \\
			\vdots & \vdots   & \vdots  \\
		\end{array}	&    &     &
		\cdots    &
		\begin{array}{ccc}
			0 & 0 & 0 \\
			\vdots  & \vdots  & \vdots \\
			\vdots & \vdots  & \vdots  \\
			0 & 0 & \alpha_1^{[N]}\tA[\Nop]_1 \\
		\end{array}	&    &        \\ 
		\hdashline
		\begin{array}{ccc}
			\alpha_1^{[1]} \ee \tbb[1]_1 & 0 & 0  \\
			\vdots  & \vdots  & \vdots   \\
			\vdots  & \vdots  & \vdots   \\
			\vdots  & \vdots  & \vdots   \\
		\end{array} & \begin{array}{ccc}
			\alpha_2^{[1]} \tA[1]_2 & 0 & 0 \\
			\alpha_2^{[1]}\ee \tbb[1]_2 & 0 & 0 \\
			\vdots  & \vdots  & \vdots  \\
			\vdots  & \vdots  & \vdots   \\
		\end{array} 
		& 
		&\begin{array}{ccc}
			0 & \alpha_1^{[2]} \ee\tbb[2]_1& 0  \\
			\vdots  & \vdots  & \vdots \\
			\vdots  & \vdots  & \vdots  \\
			\vdots  & \vdots  & \vdots \\
		\end{array} &
		\begin{array}{ccc}
			0 & 0 & 0 \\
			0 & \alpha_2^{[2]} \tA[2]_2 & 0 \\
			0 & \alpha_2^{[2]} \ee \tbb[2]_2 & 0 \\
			\vdots  & \vdots  & \vdots  \\
		\end{array}  & 
		&\cdots & \begin{array}{ccc}
			0 & 0 & \alpha_1^{[\Nop]}\ee\tbb[\Nop]_1 \\
			\vdots  & \vdots  & \vdots \\
			\vdots  & \vdots  & \vdots \\
			\vdots  & \vdots  & \vdots \\
		\end{array} &  \begin{array}{ccc}
			0 & 0 & 0 \\
			\vdots  & \vdots  & \vdots \\
			\vdots  & \vdots  & \vdots \\
			0 & 0 & \alpha_2^{[\Nop]}\tA[\Nop]_2 \\
		\end{array} &  \\
		\hdashline 
		\vdots & \vdots  & \vdots  & \vdots & \vdots &  \vdots  &  \vdots  & \vdots &\vdots  & \vdots  \\
		\hdashline 
		\begin{array}{ccc}
			\vdots  & \vdots  & \vdots   \\
			\vdots  & \vdots  & \vdots  \\
			\vdots  & \vdots  & \vdots  \\
			\alpha_1^{[1]} \ee \tbb[1]_1 & 0 & 0   \\
		\end{array} & \begin{array}{ccc}
			\vdots  & \vdots  & \vdots   \\
			\vdots  & \vdots  & \vdots \\
			\vdots  & \vdots  & \vdots \\
			\alpha_2^{[1]} \ee\tbb[1]_2 & 0 & 0 \\
		\end{array} & \begin{array}{ccc}
			\alpha_s^{[1]} \tA[1]_s & 0 & 0 \\
			\alpha_s^{[1]} \ee\tbb[1]_s & 0 & 0 \\
			\vdots & \vdots & \vdots \\
			\alpha_s^{[1]} \ee\tbb[1]_s & 0 & 0 \\
		\end{array} &\begin{array}{ccc}
			\vdots  & \vdots  & \vdots  \\
			\vdots  & \vdots  & \vdots \\
			\vdots & \vdots & \vdots \\
			0 & \alpha_1^{[2]} \ee\tbb[2]_1& 0 \\
		\end{array} & \begin{array}{ccc}
			\vdots  & \vdots  & \vdots \\
			\vdots  & \vdots  & \vdots \\
			\vdots & \vdots & \vdots \\
			0 & \alpha_2^{[2]}\ee \tbb[2]_2 & 0 \\
		\end{array} & \begin{array}{ccc}
			0 & 0 & 0 \\
			0 & \alpha_s^{[2]} \tA[2]_s & 0 \\
			0 & \alpha_s^{[2]}\ee \tbb[2]_s & 0 \\
			0 & \alpha_s^{[2]}\ee \tbb[2]_s & 0 \\
		\end{array} &\cdots & \begin{array}{ccc}
			\vdots  & \vdots  & \vdots \\
			\vdots  & \vdots  & \vdots \\
			\vdots & \vdots  & \vdots \\
			0 & 0 & \alpha_1^{[\Nop]}\ee\tbb[\Nop]_1 \\
		\end{array} &  \begin{array}{ccc}
			0 & 0 & \alpha_2^{[\Nop]}\ee\tbb[\Nop]_2 \\
			\vdots  & \vdots  & \vdots \\
			\vdots & \vdots  & \vdots \\
			0 & 0 & \alpha_2^{[N]}\ee\tbb[\Nop]_2 \\
		\end{array} &  \begin{array}{ccc}
			0 & 0 & 0 \\
			\vdots  & \vdots  & \vdots \\
			\vdots & \vdots  & \vdots \\
			0 & 0 & \alpha_s^{[\Nop]}\tA[\Nop]_s \\
		\end{array}\\
		\hline 
		\begin{array}{ccc}
			\alpha_1^{[1]} \tbb[1]_1 & 0 & 0 
		\end{array}
		& \begin{array}{ccc}
			\alpha_1^{[1]} \tbb[1]_1 & 0 & 0 
		\end{array}
		& \begin{array}{ccc}
			\alpha_s^{[1]} \tbb[1]_2 & 0 & 0
		\end{array}
		& \begin{array}{ccc}
			0 & \alpha_1^{[2]} \tbb[2]_1 & 0
		\end{array}
		& \begin{array}{ccc}
			0 & \alpha_2^{[2]} \tbb[2]_2 & 0
		\end{array}
		& \begin{array}{ccc}
			0 & \alpha_s^{[2]} \tbb[2]_s & 0
		\end{array}
		& \cdots
		& \begin{array}{ccc}
			0 & 0 & \alpha_1^{[N]} \tbb[\Nop]_1
		\end{array}
		& \begin{array}{ccc}
			0 & 0 & \alpha_2^{[N]} \tbb[\Nop]_2
		\end{array}
		& \begin{array}{ccc}
			0 & 0 & \alpha_s^{[N]} \tbb[\Nop]_s
		\end{array} \\
	\end{array}$	
}
\end{equation}

\begin{rmk}
Our construction of the extended Butcher tableau uses the OS method
\cref{eq:os_method}. The rows of the tableaux assume the
intermediate variables $\YY_{k,i}^{[\ell]}$ are ordered as they
appear in the FSRK method, i.e.,
$\YY_{1,1:\ts_{1}^{[1]}}^{[1]},\dots,\YY_{1,1:\ts_{1}^{[\Nop]}}^{[\Nop]},
\dots,$
$\YY_{s,1:\ts_{s}^{[1]}}^{[1]},\dots,\YY_{s,1:\ts_{s}^{[\Nop]}}^{[\Nop]}$. If
an operator-splitting method is constructed by composing
\cref{eq:os_godunov} and \cref{eq:os_godunovadj} over fractions of
$\Dt$, then the intermediate variables $\YY_{k,i}^{[\ell]}$ should
be re-ordered in the order they are applied to obtain an extended
Butcher tableau in the same structure as presented in
\cref{th:butcher_theorem}. If one operator-splitting stage
	of the form \cref{eq:os_godunov} is applied, the intermediate
	variables should be ordered as
	$\YY_{k,1:\ts_{k}^{[1]}}^{[1]},\dots,\YY_{k,1:\ts_{k}^{[\Nop]}}^{[\Nop]}$. If
	one operator-splitting stage of the form \cref{eq:os_godunovadj}
	is applied, the intermediate variables should be ordered as
	$\YY_{k,1:\ts_{k}^{[\Nop]}}^{[\Nop]},\dots,\YY_{k,1:\ts_{k}^{[1]}}^{[1]}$. 

Furthermore, the proof of \cref{th:butcher_theorem} shows that every FSRK tableau of the form \cref{tab:extended_tab_comp} can be written as an ARK tableau and vice-versa. We can also reorder the block rows and columns of \cref{tab:extended_tab_comp} to recover the GARK tableau of the form \cref{gark_tab}.

	For example, two-stage, second-order 2-operator-splitting
	methods form a one-parameter family of methods with free parameter
	$\OS22b$. We denote the members of this family by OS$_2$(2,2)-$\OS22b$
	and present their coefficients in \cref{tab:OS22b}.
	\begin{table}
		\centering
		\begin{tabular}{|c|c|c|}
			\hline 
			$k$ & $\alpha_k^{[1]}$ & $\alpha_k^{[2]}$ \\ \hline 
			1 & $\displaystyle \frac{2\OS22b-1}{2\OS22b -2}$   & $1-\OS22b$    \\ \hline 
			2 &  $\displaystyle - \frac{1}{2\OS22b-2}$  & $\OS22b$   \\ \hline 
		\end{tabular}
		\caption{Coefficients $\alpha_k^{[\ell]}$ for a two-stage,
			second-order, 2-OS method OS$_2$(2,2)-$\OS22b$. }
		\label{tab:OS22b}
	\end{table}
	Suppose a $2$-additive ODE is solved using the OS$_2$(2,2)-$\OS22b$
	method
	\[
	\pphi{2}{\OS22b \Dt} \circ \pphi{1}{-1/(2\OS22b-2))\Dt} \circ \pphi{2}{(1-\OS22b)\Dt} \circ \pphi{1}{(2\OS22b-1)/(2\OS22b-2)\Dt},
	\]
	where each operator is solved with an $\ts^{[\ell]}_k$-stage
	Runge--Kutta method with Butcher tableau $\begin{array}{c|c}
		\tcc[\ell]_k & \tA[\ell]_k \\
		\hline
		& \tbb[\ell]_k 
	\end{array}$. The compact version of the extended Butcher tableau constructed using \cref{th:butcher_theorem} is given in \cref{ex:OS22b_comp}, where we present only the main matrix and label the $\osgarkY{1}{1}{1:\ts_{k}^{[\ell]}}$ along the rows and columns for clarity. 
	\begin{equation} 
		\label{ex:OS22b_comp}
		\begin{array}{c|cc:cc}
			& \osgarkY{1}{1}{1:\ts_{1}^{[1]}} &  \osgarkY{1}{2}{1:\ts_{1}^{[2]}} &  \osgarkY{2}{1}{1:\ts_{2}^{[1]}}  &   \osgarkY{2}{2}{1:\ts_{2}^{[2]}} \\
			\hline 
			\osgarkY{1}{1}{1:\ts_{1}^{[1]}}  & \frac{2\OS22b-1}{2\OS22b -2} \tA[1]_1     &    &      \\
			\osgarkY{1}{2}{1:\ts_{1}^{[2]}} & \frac{2\OS22b-1}{2\OS22b -2} \ee \tbb[1]_1   &   (1-\OS22b) \tA[2]_1   \\ 
			\hdashline
			\osgarkY{2}{1}{1:\ts_{2}^{[1]}}   & \frac{2\OS22b-1}{2\OS22b -2} \ee \tbb[1]_1     &   (1-\OS22b) \ee \tbb[2]_1   &   -\frac{1}{2\OS22b-2}\tA[1]_2   \\
			\osgarkY{2}{2}{1:\ts_{2}^{[2]}} &   \frac{2\OS22b-1}{2\OS22b -2} \ee \tbb[1]_1   &  (1-\OS22b) \ee \tbb[2]_1  &  -\frac{1}{2\OS22b-2} \ee \tbb[1]_2  & \OS22b \tA[2]_2   \\
			\hline 
		\end{array}. 
	\end{equation}
	The Butcher tableau corresponds to the GARK structure in \cref{gark_tab} is given in \cref{ex:OS22b_gark}. 
	\begin{equation} 
		\label{ex:OS22b_gark}
		\begin{array}{c|cc:cc}
			& \osgarkY{1}{1}{1:\ts_{1}^{[1]}} &  \osgarkY{2}{1}{1:\ts_{2}^{[1]}} &\osgarkY{1}{2}{1:\ts_{1}^{[2]}}   &   \osgarkY{2}{2}{1:\ts_{2}^{[2]}} \\
			\hline 
			\osgarkY{1}{1}{1:\ts_{1}^{[1]}}  & \frac{2\OS22b-1}{2\OS22b -2} \tA[1]_1     &    &      \\
			\osgarkY{2}{1}{1:\ts_{2}^{[1]}}   & \frac{2\OS22b-1}{2\OS22b -2} \ee \tbb[1]_1   &    -\frac{1}{2\OS22b-2}\tA[1]_2   & (1-\OS22b) \ee \tbb[2]_1  \\ 
			\hdashline
			\osgarkY{1}{2}{1:\ts_{1}^{[2]}}  & \frac{2\OS22b-1}{2\OS22b -2} \ee \tbb[1]_1     &     & (1-\OS22b) \tA[2]_1  &   \\
			\osgarkY{2}{2}{1:\ts_{2}^{[2]}} &   \frac{2\OS22b-1}{2\OS22b -2} \ee \tbb[1]_1   & -\frac{1}{2\OS22b-2}\ee \tbb[1]_2  & (1-\OS22b) \ee \tbb[2]_1      & \OS22b \tA[2]_2   \\
			\hline 
		\end{array}. 
	\end{equation}
	We note that the Butcher tableau \cref{ex:OS22b_gark} is
	equivalent to the compact Butcher tableau \cref{ex:OS22b_comp}
	after re-ordering the intermediate variables
	$\YY_{k,i}^{[\ell]}$. In implementation, we note that the format
	of~\cref{ex:OS22b_comp} is convenient because it is intuitive to
	construct the tableau from the data row-by-row in the order in
	which they are used and also to solve for $\YY_{k,i}^{[\ell]}$
	when using the block lower-triangular form.

\end{rmk}


\Cref{thm:stabfunc}
presents the main result on the stability function of an FSRK method.

\begin{thm} \label{thm:stabfunc}
We apply the FSRK method \eqref{eq:fsrk} to the linear test equation
\begin{equation}
	\label{eq:test}
	\dv{y}{t} = \sum\limits_{\ell=1}^{N} \lam[\ell] y. 
\end{equation}
We define $\zl = \Delta t \lam[\ell]$ and the stability function of
each Runge--Kutta method used to integrate each operator
to be $R_k^{[\ell]}(z^{[\ell]})$, $\ell = 1,2,\dots,\Nop$, $k=1,2,\dots,s$. Then the
stability function $R(z^{[1]}, z^{[2]}, \dots, z^{[\Nop]})$ of the
FSRK method is given by
\begin{equation} \label{stabilityofFSRK}
	R(z^{[1]}, z^{[2]}, \dots, z^{[N]}) 
	= \prod_{k=1}^s \prod_{\ell=1}^N  R_k^{[\ell]} (\alpha_k^{[\ell]} z^{[\ell]}) .
\end{equation}

That is, \textit{the stability function of the FSRK method applied
	to~\cref{eq:test} is the product of the stability functions of the
	individual RK methods with arguments scaled by the OS method
	coefficients.}
\end{thm}

\begin{proof}
Assume that we apply a Runge--Kutta method to the operator
$\Fl[\ell]$ at stage $k$ of the FS method. We refer to this
Runge--Kutta method as RK$_{k}^{[\ell]}$ with corresponding Butcher
tableau

\begin{equation*}
	\begin{array}{c|c}
		\tcc[\ell]_k & \tA[\ell]_k \\
		\hline 
		& \tbb[\ell]_k
	\end{array}.	
\end{equation*} 
Let $R_{k}^{[\ell]} (z^{[\ell]})$ be the stability function corresponds to RK$_{k}^{[\ell]}$.

Let $y_k^{[\ell]}$ be the intermediate solution after solving
$\displaystyle \dv{y^{[\ell]}}{t} = \lambda^{[\ell]}y^{[\ell]}$ at stage $k$.

Therefore, after solving $\displaystyle \dv{y^{[1]}}{t} = \lam[1] y^{[1]}$ at
stage $1$,
\[ 
y_{1}^{[1]} = R_{1}^{[1]} (\alpha_{1}^{[1]} z^{[1]}) y_n,
\]
and after solving $\displaystyle \dv{y^{[N]}}{t} = \lam[N] y^{[N]}$ at
stage $1$,
\[ 
y_{1}^{[N]} = \left( \prod_{\ell=1}^N R_{1}^{[\ell]} (\alpha_{1}^{[\ell]} z^{[\ell]})\right) y_n.
\]
By repeating this process over all operators and stages, we can
write $y_{n+1}$ as
\[
y_{n+1} = y_{s}^{[N]} = \left(\prod_{k=1}^s \prod_{\ell=1}^N  R_{k}^{[\ell]} (\alpha_{k}^{[\ell]} z^{[\ell]})\right) y_n.
\]

\end{proof}

\begin{rmk}
\Cref{thm:stabfunc} is a generalization of simpler, lower-order
results found in~\cite{Hundsdorfer2003, ropp2005, ropp2009}. 
\end{rmk}

\begin{rmk}
\label{rmk:stabfunc}
The FSRK method \eqref{eq:fsrk} can be described using the extended
Butcher tableau \eqref{tab:extended_tab}, which has the structure of
an ARK method. Using example 4 in \cite{sandu2015}, the stability
function can also be written as

\begin{equation}\label{ark_stab_func}
	R(z^{[1]}, z^{[2]}, \dots, z^{[N]})  =  1+\left(\sum\limits_{\ell=1}^N z^{[\ell]} \bb[\ell]\right) \cdot \left(\mathbf{I}_{\mathbb{S}\times \mathbb{S}} - \left(\sum\limits_{\ell=1}^N z^{[\ell]} \mathbf{A}^{[\ell]}\right) \right)^{-1} \cdot \ee_{\mathbb{S}},
\end{equation}
where $\AAA[\ell]$ and $\bb[\ell]$ are as defined in the extended
Butcher tableau in \cref{th:butcher_theorem}, $\ee$ is the vector
of ones, and
$\displaystyle \bbS{}{}=\sum\limits_{k=1}^s \bbS{k}{\Nop} =
\sum\limits_{k=1}^s \sum\limits_{\ell=1}^\Nop
\ts_{k}^{[\ell]}$. After some linear algebra,
	\cref{stabilityofFSRK} can be recovered from
	\cref{ark_stab_func}. Similarly, equation (4.2) in
	\cite{sandu2015} presents the stability function of a GARK
	method using GARK tableau \cref{gark_tab}.
	\Cref{stabilityofFSRK} can be recovered from equation (4.2) in
	\cite{sandu2015} with the GARK tableau obtained by reordering
	\cref{tab:extended_tab_comp}. Both \cref{ark_stab_func} and
	equation (4.2) in \cite{sandu2015} have theoretical
	importance. However, for FSRK methods, the stability function
	\cref{stabilityofFSRK} is more practical in implementation.

\end{rmk}

\begin{rmk}
If we change the order of the sub-integrators, the stability
function of the FSRK method is generally changed, even without
changing the Runge--Kutta methods used for each operator, because
the coefficients $\alpha_{k}^{[\ell]}$ associated with each
sub-integrator are generally changed. This can explain observations
of different stability behaviour of numerical methods depending on
order of sub-integration, e.g.,~\cite{Torabi2014, Ropp2004}. See
also examples below.
\end{rmk}

\begin{rmk}
The choice of test equation~\cref{eq:test} assumes that the
Jacobians of each operator with respect to the solution $\yy$ are
simultaneously diagonalizable in a neighbourhood of the solution. It
is well known that this assumption may not lead to useful practical
analysis. Accordingly, more elaborate test equations
exist~\cite{Gear1974,Kvaerno2000}; however, there is no generally accepted
test equation that is considered standard at this time.
Nonetheless, \cref{eq:test} is often useful in practice and in fact
may be appropriate as a test equation for co-simulation.
\end{rmk}

\section{Numerical Examples}
\label{sec:numerical_ex}

In this section, we illustrate some of the results presented in this
paper, their implications, and how they can be used to explain various
observations and loose ends in the literature. We show how to
construct the Butcher tableau for a general FSRK method with different
RK methods for each operator and each OS stage, how stability depends
on the splitting (the choice of operators, their order of integration,
and sub-integrators). Of particular interest is how backward sub-steps
manifest themselves as holes in the stability region; we describe the
extent to which backward steps may destabilize a computation and how
such destabilization can be mitigated.

\subsection{Construction of the extended Butcher tableau}

\begin{ex}

We first present a simple example to construct a general extended Butcher
tableau. Consider the problem
\begin{equation*} 
	\dv{\yy}{t} = \Fl[1](t,\yy) + \Fl[2](t,\yy) + \Fl[3](t,\yy).
\end{equation*}
We solve the problem using a three-stage, second-order,
3-operator-splitting method OS$_3$(3,2) whose coefficients are given in
\cref{OS32-3coeff}.
\begin{table}[htbp]
	\centering
	\caption{Coefficients $\alpha_k^{[i]}$ for a three-stage,
		second-order, 3-OS method OS$_3$(3,2)}
	\begin{tabular}{|c|c|c|c|}
		\hline 
		$k$ & $\displaystyle \alpha_k^{[1]}$ &  $\displaystyle \alpha_k^{[2]}$  & $\displaystyle \alpha_k^{[3]}$\\
		\hline 
		1 & $1/3$ & $1$ & $1/4$ \\
		\hline 
		2 & $1/3$ & $-1/2$ & $1$\\
		\hline 
		3 & $1/3$ & $1/2$  & $-1/4$ \\
		\hline 
	\end{tabular}
	
	\label{OS32-3coeff}
\end{table}
The first sub-equation
$\displaystyle \dv{\yyy[1]}{t} = \Fl[1](t,\yyy[1])$ is integrated
using the forward Euler (FE), backward Euler (BE), and Heun methods at
stages $k=1,2,3$ respectively. The second sub-equation
$\displaystyle \dv{\yyy[2]}{t} = \Fl[2](t,\yyy[2])$ is integrated
using the Crank--Nicolson, BE, and FE methods at stages $k=1,2,3$
respectively. The third sub-equation
$\displaystyle \dv{\yyy[3]}{t} = \Fl[3](t,\yyy[3])$ is integrated
using the BE, BE, and FE methods at stages $k=1,2,3$ respectively. The
Butcher tableaux of these methods at each stage is given in
\cref{ex_tableaux}.

{\small 
	\begin{equation} \label{ex_tableaux}
		\begin{aligned}
			&\begin{array}{c|c}
				\tcc[1]_1 & \tA[1]_1\\
				\hline 
				& \tbb[1]_1  \\
			\end{array}= 
			\begin{array}{c|c}
				0 & 0\\
				\hline 
				& 1 \\
			\end{array}, 
			&& \begin{array}{c|c}
				\tcc[2]_1 & \tA[2]_1\\
				\hline 
				& \tbb[2]_1  \\
			\end{array}= 
			\begin{array}{c|cc}
				0 & 0 & 0 \\
				1 & 1/2 & 1/2 \\
				\hline 
				& 1/2 & 1/2  \\
			\end{array},    
			&&  \begin{array}{c|c}
				\tcc[3]_1 & \tA[3]_1\\
				\hline 
				& \tbb[3]_1  \\
			\end{array}= 
			\begin{array}{c|c}
				1 & 1\\
				\hline 
				& 1 \\
			\end{array}    \\
			&\begin{array}{c|c}
				\tcc[1]_2 & \tA[1]_2\\
				\hline 
				& \tbb[1]_2  \\
			\end{array}= 
			\begin{array}{c|c}
				1 & 1\\
				\hline 
				& 1 \\
			\end{array},  
			&& \begin{array}{c|c}
				\tcc[2]_2 & \tA[2]_2\\
				\hline 
				& \tbb[2]_2  \\
			\end{array}= 
			\begin{array}{c|c}
				1 & 1\\
				\hline 
				& 1 \\
			\end{array},
			&&  \begin{array}{c|c}
				\tcc[3]_2 & \tA[3]_2\\
				\hline 
				& \tbb[3]_2  \\
			\end{array}= 
			\begin{array}{c|c}
				1 & 1\\
				\hline 
				& 1 \\
			\end{array},    \\
			&\begin{array}{c|c}
				\tcc[1]_3 & \tA[1]_3\\
				\hline 
				& \tbb[1]_3  \\
			\end{array}= 
			\begin{array}{c|cc}
				0 & 0 & 0 \\
				1 & 1 & 0 \\
				\hline 
				& 1/2 & 1/2 \\
			\end{array},  
			&&   \begin{array}{c|c}
				\tcc[2]_3 & \tA[2]_3\\
				\hline 
				& \tbb[2]_3  \\
			\end{array}= 
			\begin{array}{c|c}
				0 & 0\\
				\hline 
				& 1 \\
			\end{array},     
			&&  \begin{array}{c|c}
				\tcc[3]_3 & \tA[3]_3\\
				\hline 
				& \tbb[3]_3  \\
			\end{array}= 
			\begin{array}{c|c}
				0 & 0\\
				\hline 
				& 1 \\
			\end{array}.    \\
		\end{aligned}
\end{equation}}

The extended Butcher tableau consists of three major sections
$\AAA[1]$, $\AAA[2]$, and $\AAA[3]$. Each matrix $\AAA[\ell]$ is of
size $\mathbb{S}\times \mathbb{S}$, where
$\mathbb{S}=\sum\limits_{k=1}^3 \sum\limits_{\ell=1}^3
\tilde{s}_k^{[\ell]} = 11$. In the following tableaux, the blue
numbers correspond to $\alpha_k^{[\ell]}\tA[\ell]_k$.

\begin{equation*}
	\begin{array}{c|c}
		\cc[1] & \AAA[1]\\
		\hline 
		& \bb[1]  \\
	\end{array}= 
	\begin{array}{c|c|cccc:ccc:cccc}
		\osgarkY{1}{1}{1} &0 & \textcolor{blue}{0} & 0 & 0  & 0  &   &   &   &   &   &   &  \\ 
		\osgarkY{1}{2}{1} & 1/3 & 1/3  & \vdots   & \vdots  & \vdots   &    &   &    &   &   &    &   \\ 
		\osgarkY{1}{2}{2} & 1/3 & 1/3 &  \vdots   & \vdots  & \vdots  &    &   &   &   &   &   &    \\ 
		\osgarkY{1}{3}{1} & 1/3 & 1/3 &    \vdots   & \vdots  & \vdots    &    &    &   &   &   &   &   \\ 
		\hdashline
		\osgarkY{2}{1}{1} & 2/3 & 1/3 &  \vdots   & \vdots  & \vdots   & \textcolor{blue}{1/3} & 0 & 0 &   &   &   &  \\ 
		\osgarkY{2}{2}{1} & 2/3 & 1/3 &   \vdots   & \vdots  & \vdots    & 1/3 & \vdots   & \vdots    &    &   &   &   \\ 
		\osgarkY{2}{3}{1} & 2/3 & 1/3 &   \vdots   & \vdots  & \vdots  & 1/3  &  \vdots   & \vdots    &   &   &   &   \\ 
		\hdashline
		\osgarkY{3}{1}{1} & 2/3 & 1/3 &  \vdots   & \vdots  & \vdots  & 1/3  &  \vdots   & \vdots    & \textcolor{blue}{0}  & \textcolor{blue}{0} & 0 & 0 \\ 
		\osgarkY{3}{1}{2} & 1 & 1/3 &   \vdots   & \vdots  & \vdots    & 1/3 &    \vdots   & \vdots & \textcolor{blue}{1/3}  & \textcolor{blue}{0} &   \vdots   & \vdots   \\ 
		\osgarkY{3}{2}{1} & 1 & 1/3 &  \vdots   & \vdots  & \vdots   & 1/3 &   \vdots   & \vdots    & 1/6   & 1/6 &   \vdots   & \vdots     \\ 
		\osgarkY{3}{3}{1} & 1 & 1/3 &  0   & 0 & 0  & 1/3 &  0 & 0  & 1/6 & 1/6  & 0  & 0 \\ 
		\hline 
		& & 1/3  & 0 & 0 & 0 & 1/3 & 0 & 0  & 1/6  & 1/6  & 0 & 0  \\ 
	\end{array}
\end{equation*}

\begin{equation*}
	\begin{array}{c|c}
		\cc[2] & \AAA[2]\\
		\hline 
		& \bb[2]  \\
	\end{array}= 
	\begin{array}{c|c|cccc:ccc:cccc}
		\osgarkY{1}{1}{1} & 0 & 0 & 0& 0  & 0  & &     &  &  &   &   &  \\
		\osgarkY{1}{2}{1} & 0 & \vdots  & \textcolor{blue}{0} & \textcolor{blue}{0} & \vdots   &   &       &   &    &    &     &    \\
		\osgarkY{1}{2}{2} & 1 & \vdots  & \textcolor{blue}{1/2} & \textcolor{blue}{1/2}&  \vdots &  &    &  &   &   &    &    \\
		\osgarkY{1}{3}{1} & 1 & \vdots  & 1/2 & 1/2 & \vdots  &   &     &    &     &    &     &    \\
		\hdashline
		\osgarkY{2}{1}{1} & 1&  \vdots  & 1/2 & 1/2 & \vdots  & 0  & 0  & 0  &   &   &    &   \\
		\osgarkY{2}{2}{1} & 1/2 & \vdots   & 1/2 & 1/2 & \vdots  &  \vdots &  \textcolor{blue}{-1/2} & \vdots &    &    &     &    \\
		\osgarkY{2}{3}{1} & 1/2 &  \vdots & 1/2 & 1/2 & \vdots  &  \vdots &-1/2 & \vdots   &    &    &     &    \\
		\hdashline
		\osgarkY{3}{1}{1} & 1/2 &  \vdots  & 1/2 & 1/2 & \vdots  & \vdots  & -1/2 &   \vdots & 0   & 0   &  0   &  0  \\
		\osgarkY{3}{1}{2} & 1/2 &  \vdots & 1/2 & 1/2 & \vdots  &  \vdots & -1/2 &  \vdots   &    &     &    &   \\
		\osgarkY{3}{2}{1} & 1/2 & \vdots  & 1/2 & 1/2 &  \vdots &  \vdots & -1/2 &  \vdots   &  \vdots &   \vdots  & \textcolor{blue}{0} & \vdots  \\
		\osgarkY{3}{3}{1} & 1 & 0  &      1/2 & 1/2      & 0 & 0 & -1/2 & 0  &  0 &  0  & 1/2 &  0  \\
		\hline 
		& &  0 & 1/2 & 1/2 & 0  & 0 & -1/2& 0 & 0 & 0 & 1/2 & 0 \\ 
	\end{array}
\end{equation*}

\begin{equation*}
	\begin{array}{c|c}
		\cc[3] & \AAA[3]\\
		\hline 
		& \bb[3]  \\
	\end{array}= 
	\begin{array}{c|c|cccc:ccc:cccc}
		\osgarkY{1}{1}{1} &  0 & 0  &  0  & 0  &  0     &   &   &      &   &   &   &     \\
		\osgarkY{1}{2}{1} &  0 & \vdots   & \vdots   & \vdots   & \vdots      &   &   &      &   &   &   &    \\
		\osgarkY{1}{2}{2} &   0 & \vdots   & \vdots   & \vdots   & 0    &   &   &      &   &   &   &    \\
		\osgarkY{1}{3}{1} &   1/4 & \vdots   & \vdots   & \vdots  & \textcolor{blue}{1/4 } &   &    &   &    &    &    &    \\
		\hdashline
		\osgarkY{2}{1}{1} & 1/4 & \vdots   & \vdots   & \vdots &  1/4 &  0 &  0  & 0  &    &    &   &    \\
		\osgarkY{2}{2}{1} &  1/4 & \vdots   & \vdots   & \vdots &  1/4 &  \vdots   & \vdots   & 0 &    &    &   &    \\
		\osgarkY{2}{3}{1} & 5/4 & \vdots   & \vdots   & \vdots &  1/4 & \vdots  & \vdots &\textcolor{blue}{1 }   &   &   &   &    \\
		\hdashline
		\osgarkY{3}{1}{1} & 5/4 & \vdots   & \vdots   & \vdots &  1/4 & \vdots  & \vdots  &  1  & 0   &  0  & 0 &  0  \\
		\osgarkY{3}{1}{2} & 5/4 & \vdots   & \vdots   & \vdots &  1/4 & \vdots  & \vdots  &  1  &  \vdots   & \vdots   & \vdots  & \vdots     \\
		\osgarkY{3}{2}{1} & 5/4 & \vdots   & \vdots   & \vdots &  1/4 & \vdots  & \vdots  &  1  &  \vdots   & \vdots   & \vdots  & 0  \\
		\osgarkY{3}{3}{1} &  5/4 & 0  & 0   & 0 &  1/4 & 0  & 0  &  1   &  0 &  0 & 0  & \textcolor{blue}{ 0} \\
		\hline 
		& & 0  & 0   & 0 &  1/4 & 0  & 0  &  1   &  0 &  0    & 0 & -1/4  \\ 
	\end{array}
\end{equation*}

The compact tableau is given below. The blue numbers correspond to
$\alpha_k^{[\ell]}\tA[\ell]_k$.

\begin{equation*}
	\resizebox{\textwidth}{!}{%
		$
		\begin{array}{ccc|c}
			\cc[1] & \cc[2] & \cc[3] & \mathbf{A} \\
			\hline 
			& & & \mathbf{b}  \\
		\end{array} 	= \begin{array}{ccc|cccc:ccc:cccc}
			0  & 0  & 0   & \textcolor{blue}{ 0} &    &   &       &   &   &      &   &   &   &    \\
			1/3  & 0  & 0    & 1/3   &  \textcolor{blue}{0 }  & \textcolor{blue}{ 0} &       &   &   &    &   &   &   &    \\
			1/3  & 1  &  0  & 1/3  & \textcolor{blue}{1/2 }  & \textcolor{blue}{ 1/2 } &       &   &   &      &   &   &   &    \\
			1/3  & 1  &  1/4  & 1/3  & 1/2  & 1/2  & \textcolor{blue}{1/4 } &  &    &   &    &    &    &    \\
			\hdashline
			2/3  & 1  & 1/4  & 1/3  & 1/2  & 1/2   & 1/4  &  \textcolor{blue}{ 1/3} &    &   &    &    &   &    \\
			2/3  & 1/2  & 1/4  &  1/3  & 1/2  & 1/2   & 1/4 & 1/3   & \textcolor{blue}{-1/2 }   &   &    &    &   &    \\
			2/3  & 1/2  & 5/4  &  1/3  & 1/2  & 1/2   & 1/4 & 1/3    &  -1/2   &\textcolor{blue}{ 1 }   &   &   &   &    \\
			\hdashline
			2/3 & 1/2  & 5/4  &   1/3  & 1/2  & 1/2   & 1/4 & 1/3    &  -1/2   &  1   &  \textcolor{blue}{0 }  &  \textcolor{blue}{ 0}  &   &    \\
			1 & 1/2  & 5/4   &    1/3  & 1/2  & 1/2   & 1/4 & 1/3    &  -1/2   &  1    &  \textcolor{blue}{1/3 }  &   \textcolor{blue}{0 }  &   &    \\
			1 & 1/2  & 5/4   &   1/3  & 1/2  & 1/2   & 1/4 & 1/3    &  -1/2   &  1   &   1/6  &   1/6   &  \textcolor{blue}{0 } &   \\
			1 & 1  &  5/4  &    1/3  & 1/2  & 1/2   & 1/4 & 1/3    &  -1/2   &  1   &   1/6  &   1/6    &  1/2  & \textcolor{blue}{0 } \\
			\hline 
			& & &    1/3  & 1/2  & 1/2   & 1/4 & 1/3    &  -1/2   &  1   &   1/6  &   1/6    &  1/2   &  -1/4 \\ 
		\end{array} $
	}
\end{equation*}

\end{ex}

\subsection{Construction of the stability function and linear
stability analysis}

The next example demonstrates how to construct the stability function
from~\cref{thm:stabfunc} and how linear stability analysis can be used
to understand some observed stability behaviour when an ODE is solved
via an FSRK method. This example also illustrates the how the
stability behaviour can depend on the splitting.

\begin{ex} \label{ex:linear_stab_ex}
Consider the differential equation 
\begin{equation} \label{eq:stability_ex}
	\dv{y}{t} = -20y = \lambda^{[1]} y + \lambda^{[2]} y, \enskip y(0) =  1. 
\end{equation} 
We apply the second-order Strang--Marchuk splitting method to solve
\cref{eq:stability_ex}, where the first sub-equation
$\displaystyle \dv{y^{[1]}}{t} = \lambda^{[1]} y
^{[1]} $ is solved using Heun's method, and the second
sub-equation $\displaystyle \dv{y^{[2]}}{t} = \lambda^{[2]} y
^{[2]} $ is solved
using the two-stage, second-order, L-stable singly diagonally implicit
Runge--Kutta method (SDIRK(2,2)).

The stability function for Heun's method is 
\begin{equation*} \label{eq:heun_stab_func}
	R_{\text{Heun}}(z) = 1+ z + \frac{z^2}{2}.
\end{equation*}

The stability function for the SDIRK(2,2) method is 

\begin{equation*} \label{eq:sdirk2o2_stab_func}
	R_{\text{SDIRK(2,2)}}(z) = \frac{z - 2\gamma z +1 }{(\gamma z
		-1)^2}, \quad \gamma = \frac{2-\sqrt{2}}{2}.
\end{equation*}

The stability function for the described FSRK method is 
\begin{equation*} 
	R(z^{[1]}, z^{[2]}) = \left[R_{\text{Heun}}\left(\frac{1}{2}z^{[1]}\right)\right]^2 R_{\text{SDIRK(2,2)}}(z^{[2]}).
\end{equation*}

We now consider three different splittings: 

\begin{itemize}
	\item Case 1 (50-50 split): $\lambda^{[1]} = -10$ and
	$\lambda^{[2]} = -10$. In this case, $z^{[1]} = z^{[2]} =
	-10\Dt$. Let $z = -\Dt$, $z^{[1]} = z^{[2]} = 10z$. The stability
	region is given by
	
	\begin{equation*}\label{eq:strang_heun_sdirk2o2_stab_func1}
		R_{\text{50-50}}(z) = \frac{(25z^2/2 + 5z + 1)^2(10z - 20\gamma z + 1)}{(10\gamma z - 1)^2}
	\end{equation*}
	

	\item Case 2 (10-90 split): $\lambda^{[1]} = -2$ and
	$\lambda^{[2]} = -18$.  Let $z = -\Dt$,
	$z^{[1]} =2z, z^{[2]} = 18z$.  The stability region is given by
	
	\begin{equation*}\label{eq:strang_heun_sdirk2o2_stab_func2}
		R_{\text{10-90}}(z) = \frac{(z^2/2 + z + 1)^2(18z - 36\gamma z + 1)}{(18\gamma z - 1)^2}
	\end{equation*}
	

	\item Case 3 (90-10 split): $\lambda^{[1]} = -18$ and
	$\lambda^{[2]} = -2$.  Let $z = -\Dt$, $z^{[1]} =18z, z^{[2]} = 2z$.
	The stability region is given by
	
	\begin{equation*}\label{eq:strang_heun_sdirk2o2_stab_func3}
		R_{\text{90-10}}(z) = \frac{(81/2 \, z^2 + 9z + 1)^2(2z - 4\gamma z + 1)}{(2\gamma z - 1)^2}
	\end{equation*}


	\begin{figure}
		\centering
		\includegraphics[width=5in]{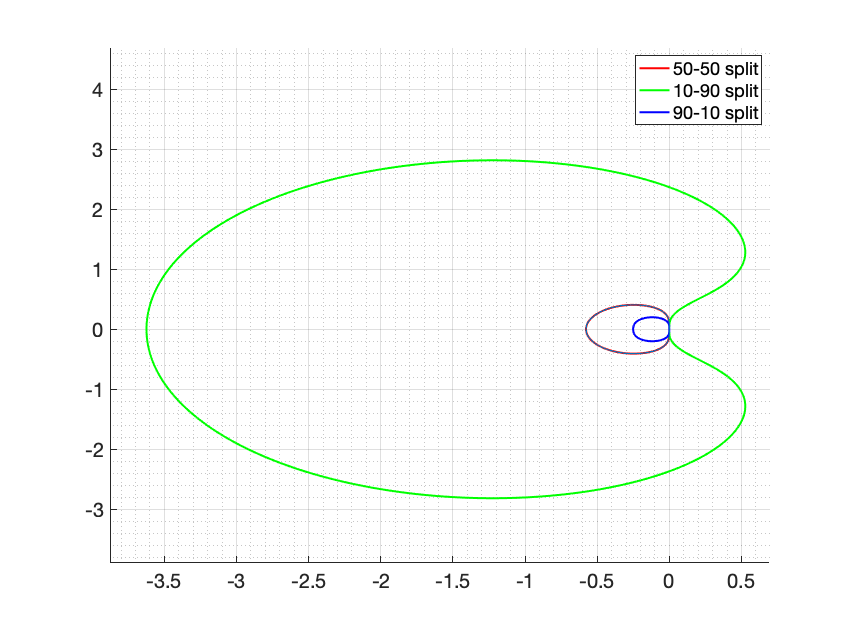}
		\caption{The interior region of each curve is the stability region
			for the Strang--Marchuk splitting method applied with the Heun
			and SDIRK(2,2) methods.}
		\label{fig:stability_linear_ex}
	\end{figure}

\end{itemize}

The stability regions $|R_{\text{50-50}}(z)| < 1$,
$|R_{\text{10-90}}(z)| < 1$, and $|R_{\text{90-10}}(z)| < 1$ are the
interior regions of the curves in \cref{fig:stability_linear_ex}. 
The figure confirms the common expectation that the stability of a
splitting method is improved when it is possible to treat the stiff
part of an ODE with an L-stable method.
\end{ex}

\subsection{The Brusselator problem}
\begin{ex}
In \cite{ropp2005}, the instability of the Brusselator problem is
explored when solved using the second-order Strang
operator-splitting method with the trapezoidal rule for the
diffusion term and CVODE~\cite{Hindmarsh2005} for the reaction term. To
analyze the stability in the language of FSRK, we recreate the
instability observed in~\cite{ropp2005} using the Strang
operator-splitting method with Heun's method as sub-integrators and
explain it using the stability function established in
\cref{thm:stabfunc}.

The Brusselator problem is defined as follows
\begin{subequations} \label{eq:brusselator}
	\begin{align}
		\pdv{T}{t} &= D_1 \pdv[2]{T}{x} + \alpha - (\beta+1)T + T^2C, \\
		\pdv{C}{t} &= D_2 \pdv[2]{C}{x} + \beta T -T^2C, 
	\end{align}
\end{subequations}
where $T$ and $C$ represent concentrations of different chemical
species. In \cite{ropp2005}, the authors considered parameter values
of $\alpha = 0.6$, $\beta =2$, and
$\displaystyle D_1= D_2 = \frac{1}{40}$, with boundary conditions
$T(0,t) = T(1,t) = \alpha$ and
$\displaystyle C(0,t) = C(1,t) = \frac{\beta}{\alpha}$ and initial
conditions $T(x,0) = \alpha + x(1-x)$ and
$\displaystyle C(x,0) = \frac{\beta}{\alpha} + x^2(1-x)$. 
\Cref{eq:brusselator} is split according to diffusion and reaction
as
\begin{equation*}
	\begin{aligned}
		\pdv{T^{[1]}}{t} &= D_1 \pdv[2]{T^{[1]}}{x}, \\
		\pdv{C^{[1]}}{t} &= D_2 \pdv[2]{C ^{[1]}}{x} 
	\end{aligned}
\end{equation*}
and 
\begin{equation*}
	\begin{aligned}
		\pdv{T^{[2]}}{t} &= \alpha - (\beta+1)T^{[2]} + (T^{[2]})^2C^{[2]},  \\
		\pdv{C^{[2]}}{t} &= \beta T^{[2]} -(T^{[2]})^2C^{[2]}.
	\end{aligned}
\end{equation*}

A reference solution for $t\in [0,80]$ is computed using the MATLAB
parabolic and elliptic PDE solver {\tt pdepe}. We decreased the
spatial meshsize $\Delta x$ and adjusted the absolute and relative
tolerances for the solver until there were at least $6$ matching
digits between successive approximations at 32,000 and 800 uniformly
distributed points in space and time, respectively.

For our experiments, the spatial derivatives are discretized using
central finite differences on a uniform grid on the interval
$x\in [0,1]$. 
The ensuing method-of-lines ODEs are then solved using the Strang
operator-splitting method with Heun's method applied to both the
reaction and diffusion systems.
%
%
%
%
%
The unstable behavior is depicted in \cref{fig:unstableStrang}. 

\begin{figure}[htbp] 
	\centering
	\includegraphics[width = \textwidth]{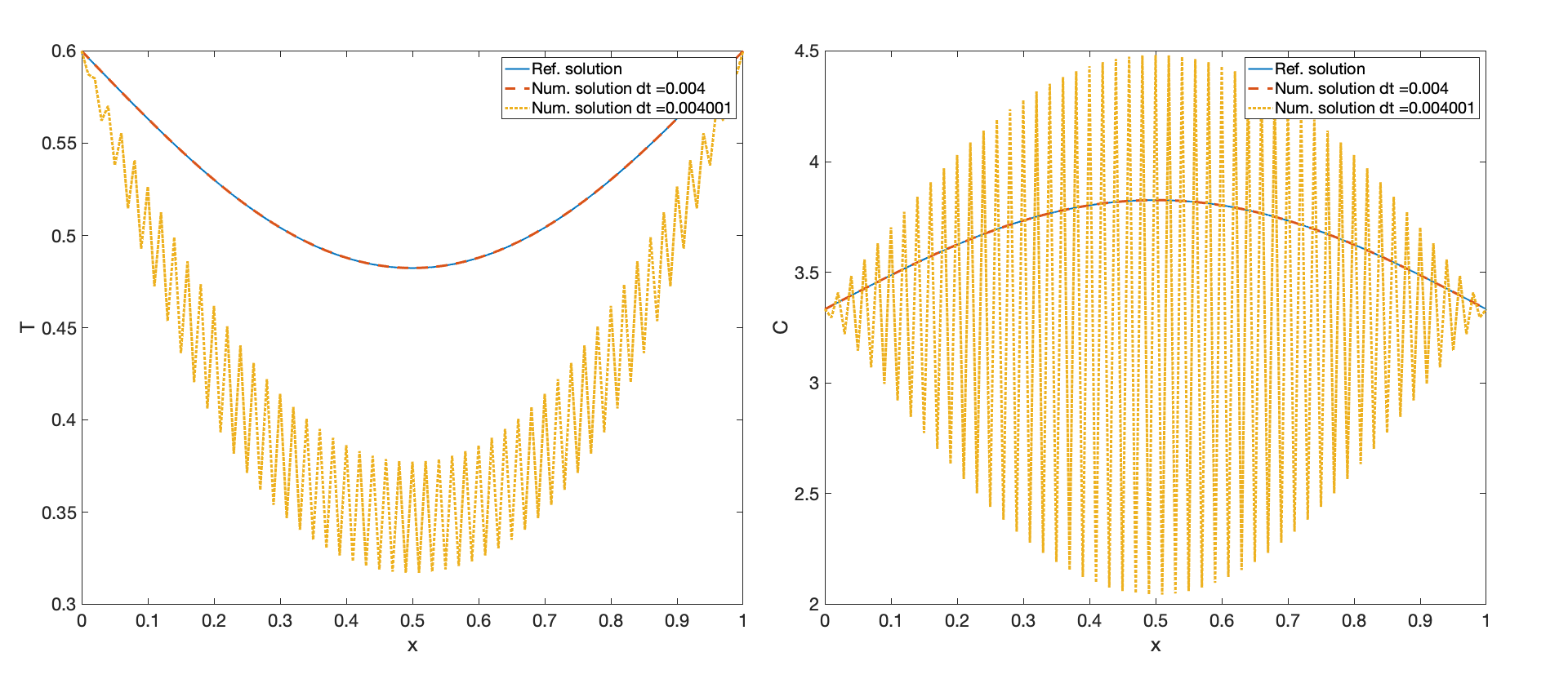}
	\caption{Solution at $t=80$ using the Strang (Heun+Heun)
			method with $\Dt = 0.004$ and $\Dt = 0.004001$ and
			$\Dx = 0.01$. We note that when solved with $\Dt = 0.004$, the
			instability that occurs for $\Dt = 0.004001$ is resolved. }
	\label{fig:unstableStrang}
\end{figure}
To understand this unstable behavior, we consider the stability
function of the Strang (Heun+Heun) method using \cref{thm:stabfunc}:
\begin{equation} \label{eq:strang_heun_stab_func}
	\begin{aligned}
		R(z^{[1]}, z^{[2]}) & = \left[R_{\text{Heun}} \left(\frac{1}{2}z^{[1]} \right) \right]^2R_{\text{Heun}} (z^{[2]})  \\	&  = \left(1+\frac{1}{2}z^{[1]}+\frac{1}{8}(z^{[1]})^2 \right)^2\left( 1+z^{[2]}+\frac{1}{2}(z^{[2]})^2 \right),
	\end{aligned}
\end{equation}
where $z^{[1]} = \lambda^{[1]} \Dt$ and $z^{[2]} = \lambda^{[2]}
\Dt$. We compute the eigenvalues of the Jacobian matrices of the
diffusion and reaction system. The Jacobian matrix of the diffusion
system is
\[
J_{\text{Diffusion}}= \begin{bmatrix}
	D_1 M & 0 \\
	0 & D_2 M \\
\end{bmatrix},
\] 
where $\displaystyle M = \frac{1}{\Dx^2}\begin{bmatrix}
	0 & \cdots & \cdots & \cdots & 0  \\
	1 & -2 & 1 &  &  \\
	& \ddots & \ddots & \ddots &  \\
	&  & 1 & -2 & 1 \\
	0 & \cdots & \cdots & \cdots & 0  \\
\end{bmatrix}$
. The Jacobian matrix of the reaction system is
\[
J_{\text{Reaction}}= 
\left[
\begin{array}{ccc:ccc}
	0 & \cdots & 0  & 0 & \cdots & 0 \\
	&-(\beta+1)+2T_iC_i & & & T_i^2 & \\
	0 & \cdots & 0  & 0 & \cdots & 0 \\
	\hdashline 
	0 & \cdots & 0  & 0 & \cdots & 0 \\
	& \beta-2T_iC_i & & & -T_i^2 & \\
	0 & \cdots & 0  & 0 & \cdots & 0 \\
\end{array} 
\right]
\]
A plot of the eigenvalues for $t\in[0,80]$ and $\Dx = 0.01$ is shown
in \cref{fig:eigenplotNx101}.
\begin{figure}[htbp] 
	\centering
	\includegraphics[width = \textwidth]{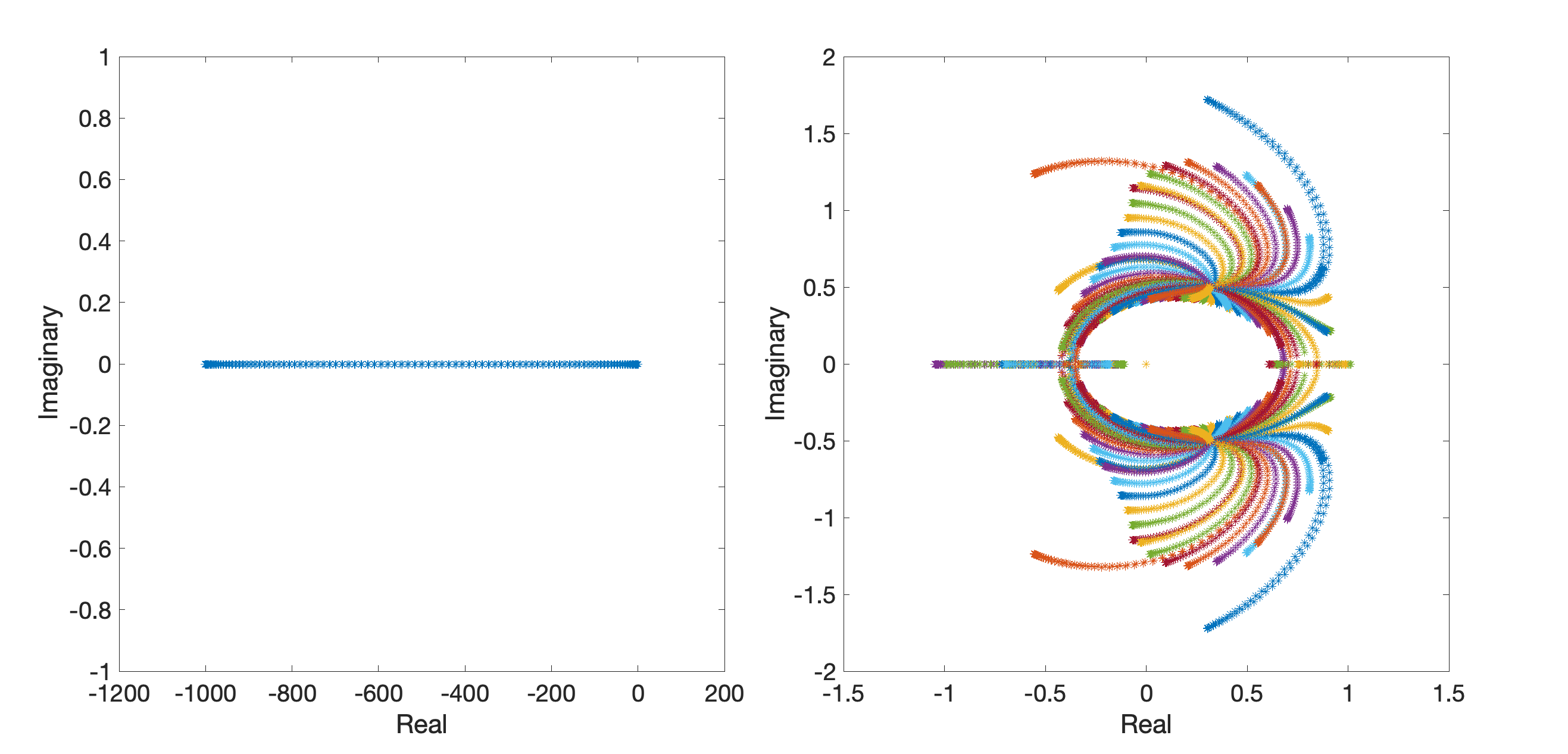}
	\caption{Eigenvalue of the Jacobian matrices $J_{\text{Diffusion}}$ and $J_{\text{Reaction}}$. }
	\label{fig:eigenplotNx101}
\end{figure}

Based on the distribution of the eigenvalues and the general shape of
the stability region of the Strang (Heun+Heun) OS method, we choose
$\lambda^{[1]} = -1000.75$, which is the most negative eigenvalue of
the diffusion system, and $\lambda^{[2]} = -1.047$, which is the
eigenvalue with the most negative real component of the reaction
system. Because the ratio of these two extreme eigenvalues is
approximately 1000, we let 
$z^{[2]} = 0.001z^{[1]}$. The stability function
\cref{eq:strang_heun_stab_func} can be written as
\begin{equation*}
	R(z) = \left(1+\frac{1}{2}z + \frac{1}{8}z^2 \right)^2 \left(1+0.001z + \frac{1}{2}(0.001z)^2 \right).
\end{equation*}
Based on this stability region, we estimate the largest $\Dt$ that
produces a stable solution with the Strang (Heun+Heun) method is
$\Delta t = 0.004$, agreeing with numerical experiments as shown in \cref{fig:unstableStrang}.



Linear stability regions cannot generally be expected to accurately
predict the step-size restriction for stability. However, they can be
used to qualitatively compare different FSRK methods. For example,
\cite{ropp2005} reported that integrating the diffusion operator with
an L-stable RK method can better control high wave-number
instability. Our analysis does not directly apply to this
	situation because CVODE was used as the sub-integrator for the
	reaction operator. However, if the reaction operator is treated with
	an RK method, the stability regions for FSRK methods can offer
	insight into this observation, as we now discuss.

We consider a family of SDIRK methods with the following Butcher tableau:
\begin{equation}\label{eq:sdirk}
	\begin{array}{c|cc}
		\gamma & \gamma & 0  \\
		1-\gamma & 1-2\gamma & \gamma \\ \hline 
		& 1/2 & 1/2
	\end{array}, 
\end{equation}
where $\gamma$ is a free parameter.  When $\gamma = 1/2$, the
resulting SDIRK method is an A-stable, second-order accurate
method. When $\gamma = 1+1/\sqrt{2}$, the resulting SDIRK method is an
L-stable, second-order method. We solve the Brusselator problem
\cref{eq:brusselator} again using the Strang splitting method. The
reaction operator is solved with Heun's method, and the diffusion
operator is solved in two different ways: once with the A-stable SDIRK
method ($\gamma = 1/2$) and then with the L-stable SDIRK method
($\gamma = 1+1/\sqrt{2}$). \Cref{fig:instability_plot2} confirms that
using an L-stable method with the step-size $\Dt=0.02$ improves the
stability of the solution.
For the parameter values used, the stability region for the FSRK
method that uses the A-stable SDIRK method has a negative real
intercept of $z\approx -2008$, whereas it is easy to show,
	e.g., using~\cref{thm:stabfunc} or \cref{rmk:stabfunc}, that the FSRK
	method that uses the L-stable SDIRK method is in fact A-stable
(despite the use of an ERK method as a sub-integrator).




\begin{figure}[htbp]
	\centering	\includegraphics[width=5in]{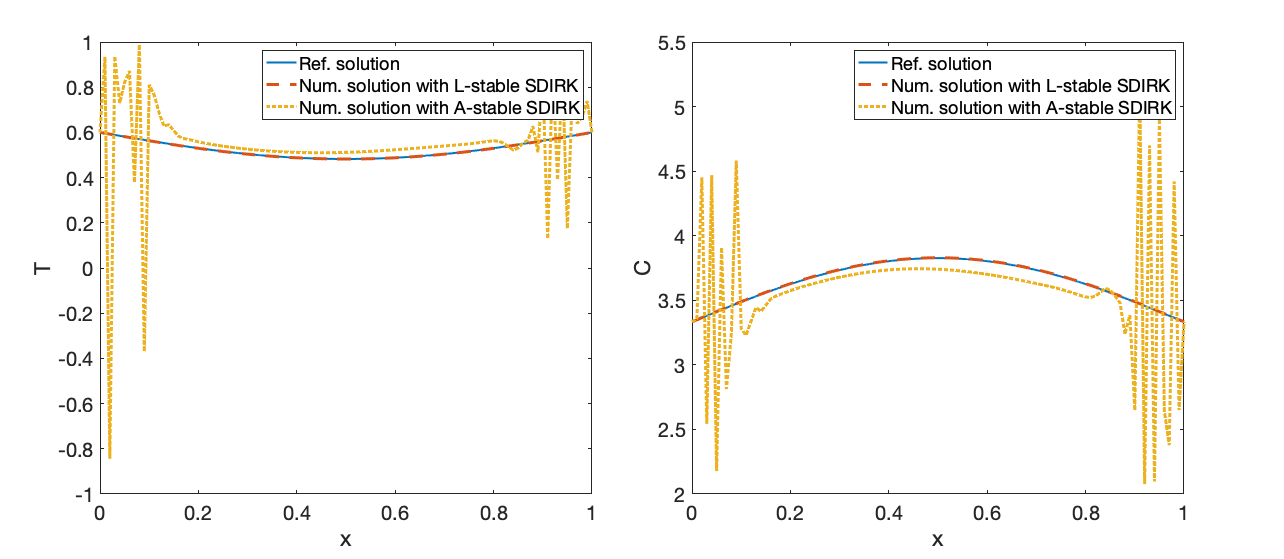}
	\caption{Solutions at $t=80$ using an A-stable SDIRK method ($\gamma=1/2$) and an L-stable SDIRK method ($\gamma = 1+1/\sqrt{2}$) with $\Dt = 0.2$. }
	\label{fig:instability_plot2}
\end{figure}


\end{ex}

\subsection{Stability regions of FSRK with negative coefficients}

OS methods of order three or higher require
backward-in-time sub-steps in each operator during the
integration~\cite{goldman1996}. There is the potential for backward
steps to create a hole in the stability region and undermine the
stability of the method in practice. We give an example of this
phenomenon in \cref{ex:hole_in_stab_region}.


\begin{ex}\label{ex:hole_in_stab_region}

Consider the differential equation
\begin{equation} \label{eq:ex_hole}
	\dv{y}{t} =  \lam[1] y+  \lam[2] y, \enskip y(0) = y_0.
\end{equation} 
Suppose we solve the ODE \cref{eq:ex_hole} using the third-order
accurate Ruth operator-splitting method whose coefficients are given
in \cref{tab:Ruthcoeff}.
\begin{table}[htbp]
	\centering
	\caption{Coefficients $\alpha_k^{[i]}$ for the Ruth method}
	\begin{tabular}{|c|c|c|}
		\hline 
		$k$ & $\displaystyle \alpha_k^{[1]}$ &  $\displaystyle \alpha_k^{[2]}$  \\
		\hline 
		1 & $7/24$  & $2/3$ \\
		\hline 
		2 & $3/4$  & $-2/3$\\
		\hline 
		3 & $-1/24$  & $1$ \\
		\hline 
	\end{tabular}
	\label{tab:Ruthcoeff}
\end{table}
The first operator $\displaystyle \dv{y}{t}^{[1]} = \lam[1]y^{[1]}$
is solved with the three-stage, third-order explicit Runge--Kutta
method due to Kutta (RK3), and the second operator
$\displaystyle \dv{y}{t}^{[2]} = \lam[2]y^{[2]}$ is solved with
SDIRK(2,3) from \cref{eq:sdirk} with $\gamma = (3+\sqrt{3})/6$. In
the case where $\lam[1] = \lam[2]$, the stability function for
$z = \lam[1] \Dt = \lam[2]\Dt$ is
\begin{equation}\label{eq:stab_func_neg_coeff}
	\begin{aligned}
		R(z)  &= R_{\text{RK3}} \left(\frac{7}{24}z \right) R_{\text{RK3}} \left(\frac{3}{4}z \right) R_{\text{RK3}} \left(-\frac{1}{24}z \right)  \cdot \\
		& \qquad R_{SDIRK(2,3)} \left(\frac{2}{3}z \right)
		R_{SDIRK(2,3)} \left(-\frac{2}{3}z \right) R_{SDIRK(2,3)}
		\left(1z \right).
	\end{aligned}
\end{equation}
The stability function for the SDIRK(2,3) method is 
\[
R_{SDIRK(2,3)}(z) = 1 - \frac{z^2(2\gamma - 1)}{2(\gamma z-1)^2} - \frac{z}{\gamma z -1},
\]
from which we see there is a singularity
in~\cref{eq:stab_func_neg_coeff} at $z = 1/(\alpha_k^{[2]}
\gamma)$. Such singularities are located in the right-half of the
complex plane when $\alpha_k^{[2]}> 0$. When $\alpha_k^{[2]}< 0$,
however, the singularity is located in the left-half of the complex
plane. In particular, for $\alpha_k^{[\ell]}= -2/3$, the singularity
is at $ z \approx -1.9$ and results in a hole in the main stability
region as shown in \cref{fig:hole_in_stab_region}. In practice, an
unfortunate combination of \textit{any} eigenvalue $\lambda$ and
$\Dt$ such that $ z = \lambda \Dt \approx -1.9$ would lead to an
unstable step and may explain why negative steps have been generally
eschewed in practice for non-reversible
problems~\cite{Sornborger1999}.  To mitigate this behavior, one
could use the implicit method on operators with small negative
coefficients $\alpha_k^{[\ell]}$. When $\alpha_k^{[\ell]}$ is
sufficiently small, the singularity would be located outside of the
stability region. For example, when SDIRK(2,3) is applied to the
first operator 
and the RK3 is applied to the second operator, the singularity in
the left-half of the complex plane is located near $z = -30.43$,
which is outside the stability region defined
by~\cref{eq:stab_func_neg_coeff} with subscripts RK3 and SDIRK(2,3)
interchanged.


\begin{figure}
	\centering
	\includegraphics[width=5in]{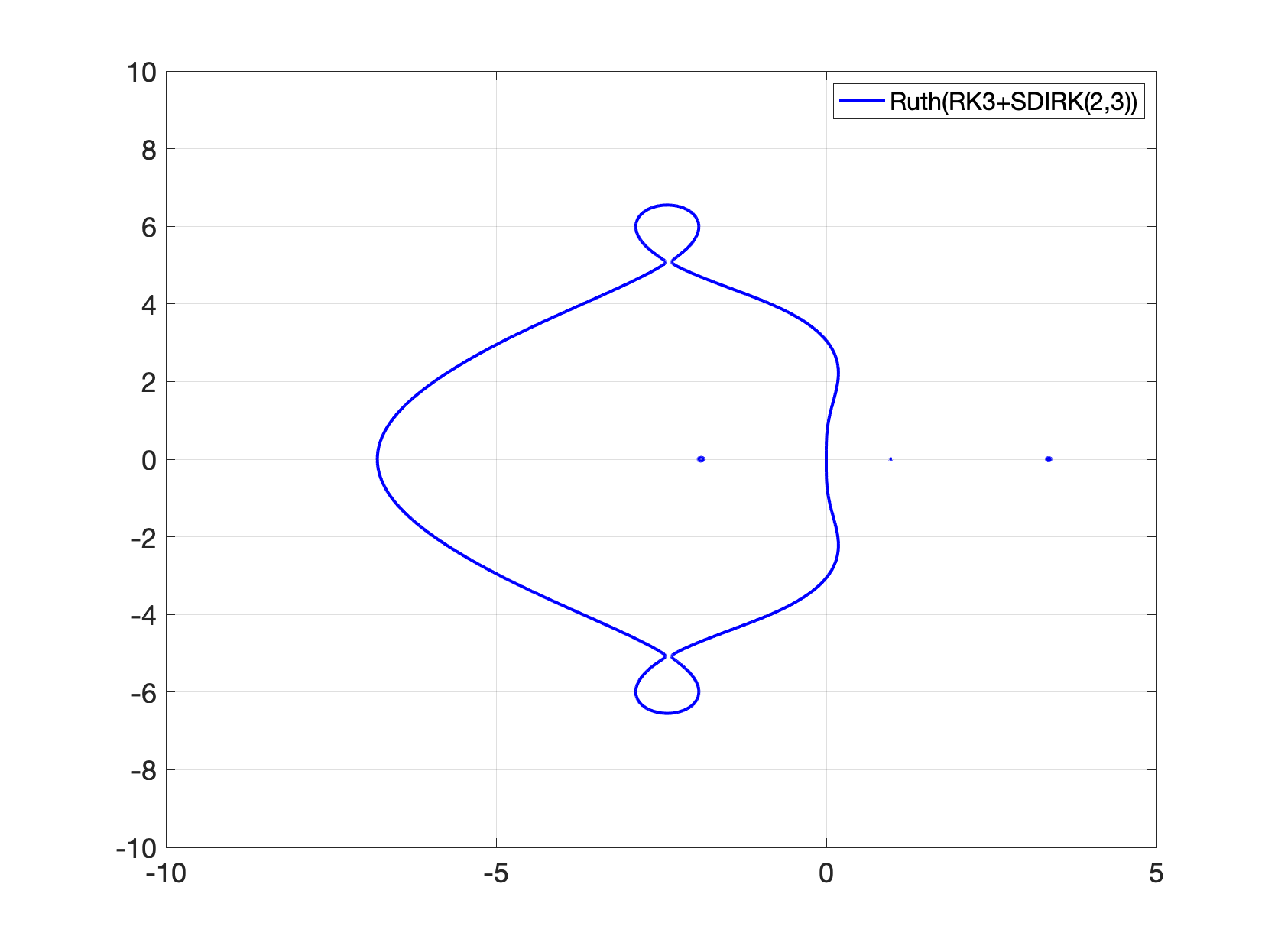}
	\caption{The stability region for the Ruth (RK3+SDIRK2O3)
		method is the interior of the large contour excluding the
		hole near $z = -1.9$. }
	\label{fig:hole_in_stab_region}
\end{figure}
\end{ex}

\section{Conclusions and future work}
\label{sec:conclusions}

We have shown how FSRK methods can be systematically represented using
Butcher tableaux within the framework of GARK methods. This
representation allows us to immediately study their stability
properties using an established framework and has further allowed us
to provide an informative interpretation of the stability function of
an FSRK method in terms of the splitting coefficients, the choice of
ordering of the operators, and the underlying RK
sub-integrators. These tools enable a systematic explanation and
understanding of common observations of FSRK methods in the literature
that have hitherto only been given as special cases. In particular, we
are able to more clearly understand the role of negative splitting
coefficients in the overall stability of an FSRK method.  The analysis
presented in this paper also provides a unified means to
develop new OS methods favourable properties. The development of such
methods is the subject of future work.

\section{CRediT author statement}


{\bf Raymond J. Spiteri:} Conceptualization, Methodology, Formal Analysis, Resources, Writing-Original Draft, Writing-Reveiw \& Editing, Supervision, Project administration, Funding acquisition 

{\bf Siqi Wei:} Methodology, Software, Validation, Formal analysis, Investigation, Data Curation, Visualization, Writing-Original Draft, Writing-Reveiw \& Editing

\newpage 

\bibliographystyle{unsrtnat}
\bibliography{FSRK_arxiv}

\end{document}